\documentclass[]{amsart} 

\def\bc{\mathbb{C}} 
\def\br{\mathbb{R}}
\def\bn{\mathbb{N}}
\def\bz{\mathbb{Z}}
\def\f{\mathbb{F}}

\newcommand{\com}[1]{#1^{\prime}}
\newcommand{\bcom}[1]{#1^{\prime\prime}}
\newcommand{\inner}[2]{\langle #1,#2\rangle} 
\newcommand{\sumjn}{\sum_{j=1}^{n}}
\newcommand{\qr}{Q^{\rm r}_{{\rm max}}(R)}
\newcommand{\lend}[2]{{\rm End}_{#1}(#2)}
\newcommand{\biend}[2]{{\rm Biend}_{#1}(#2)}

\newcommand{\lhom}[3]{{\rm Hom}_{#1}(#2,#3)}
\newcommand{\ltr}{{\rm Hom}_{\com{C},R}(V^*,U^*)}
\newcommand{\etr}{{\rm End}_{\com{C},R}(U^*)}

\newcommand{\ann}[2]{{\rm ann}_{#1}(#2)}
\newcommand{\im}{{\rm im}\,}
\newcommand{\lm}[1]{{_{#1}{\mathcal{M}}}}
\newcommand{\rmod}[1]{{\mathcal{M}}_{#1}}

\newtheorem{theorem}{Theorem}[section]
\newtheorem{lemma}[theorem]{Lemma}
\newtheorem{co}[theorem]{Corollary}
\newtheorem{pr}[theorem]{Proposition}
\theoremstyle{remark}
\newtheorem{re}[theorem]{Remark}
\theoremstyle{definition}
\newtheorem{de}[theorem]{Definition}

\numberwithin{equation}{section}

\begin{document}

\title[]{Bicommutants and Arens regularity} 
\author{Bojan Magajna} 
\address{Department of Mathematics\\ University of Ljubljana\\
Jadranska 21\\ Ljubljana 1000\\ Slovenia}
\email{Bojan.Magajna@fmf.uni-lj.si}

\thanks{Acknowledgment. The author is grateful to Matej Bre\v sar for his comments on the paper and to Peter \v Semrl
for the discussion concerning Theorem \ref{thA}.} 




\begin{abstract}Let $C$ and $R$ be unital rings,  and $Z$ an injective cogenerator for  right $C$-modules. 
For an $R,C$-bimodule $U$ let  $U^*=\lhom{C}{U}{Z}$, $\com{R}=\lend{R}{U}$  and
$\biend{R}{U}=\lend{\com{R}}{U}$, the biendomorphism ring of $U$.
Under suitable requirements on $U$ we show that $B:=\biend{R}{U}$ can be identified with a subring of $\tilde{B}:=\biend{R}{U^*}$,
study conditions for the reverse inclusion and density of $B$ in $\tilde{B}$. In the case $C$ is contained in the center
of $R$ we describe $\biend{R}{R^*}$ in terms of the Arens products in $R^{**}$  and study  Arens regularity of $R$
in the context of duality of modules. We characterize Arens regular algebras over fields.
\end{abstract}

\maketitle

\section{Introduction}
The notion of a centralizer $\com{R}$ of a subring $R$ in a ring $L$ (that is, the set of all elements of $L$ commuting 
with all elements of $R$) is fundamental in several areas of mathematics. The double centralizer ring $\bcom{R}$ of $R$ 
has continuously attracted attention in algebra (see e.g. \cite{BB}, \cite{CL} and references there
for recent results in this direction) and in functional analysis (\cite{BS} is a
recent example). For an algebra $R$ of operators on a vector space $U$ over a field $\f$ 
the centralizer $\com{R}=\lend{R}{U}$ of $R$ in $L=\lend{\f}{U}$ is usually called the commutant and the 
bicommutant  is just the ring
$\bcom{R}=\biend{R}{U}=\lend{\com{R}}{U}$ of {\em $R$-biendomorphisms} of $U$

Regarding $U$ as a left $R$-module,  the dual space $U^*$ is a right $R$-module by $\rho a:=
\rho\circ a$
($a\in R$, $\rho\in U^*$).  Clearly the adjoint $f^*$ of each endomorphism $f\in \lend{R}{U}$ acts as an endomorphism
of $U^*$, but in general $\lend{R}{U^*}$ also contains many elements which are not adjoint to any linear map $f$ on $U$
(if $U$ is infinite dimensional). 
Therefore for a map $g\in B:=\biend{R}{U}$ the adjoint $g^*$ is not necessarily an element of $\tilde{B}:=\biend{R}{U^*}$. 
For a right $R$-module, such as $U^*$, it is convenient to let biendomorphisms to act from the right,
so that $U^*$ is a right module over $\tilde{B}$,  and to take the ring multiplication in $\tilde{B}$ to 
be the reverse of the composition of maps, so that $R$ can be regarded as a subring of $\tilde{B}$. If in
$\lend{\f}{U^*}$ we reverse the composition of maps, then the involution $*$ is a ring
monomorphism $\lend{\f}{U}\to\lend{\f}{U^*}$ and may be regarded as an inclusion. In this way 
the inclusion 
\begin{equation}\label{000}\biend{R}{U}\subseteq\biend{R}{U^*}\end{equation}
makes sense, although it does not hold in general. Perhaps surprisingly, it turns out that (\ref{000}) and even the 
reverse inclusion hold under the conditions which are not very restrictive. Very special cases of the inclusion 
$\biend{R}{U^*}\subseteq\biend{R}{U}$ are known from functional analysis \cite{HW}
(where susch an inclusion means weak*-continuity of certain maps) and in \cite{M} the author (motivated by a problem
concerning derivations) considered the case when $R$ is generated by a single operator. 

Here we study more general algebras. Although we are not striving for maximal generality, it turns out that the algebraic tools we 
need work equally well also for modules instead of just vector spaces. So we will consider $R,C$-bimodules $U$
over a pair of unital rings $R,C$.  The dual $U^*$ is defined as $U^*=\lhom{C}{U}{Z}$, where
$Z$ is  an injective cogenerator in the category of right $C$-modules.
(The definition of such modules is recalled below.)
We will show that if $U$ as a left $R$-module is 
torsionless and trace accessible (that is, $TU=U$, where $T$ is the trace
ideal of $U$ in $R$; for example, $U$ may be a projective $R$-module or a generator), then (\ref{000}) holds.  The reverse inclusion
does not hold in general, but if $U$ is faithful and flat as a left $R$-module then 
\begin{equation}\label{00}\biend{R}{(U^*)^G}\subseteq\biend{R}{U^{(G)}} \ (\cong\biend{R}{U})\end{equation} for all sufficiently large 
exponents $G$. (Here $X^G$ denotes the direct product and $X^{(G)}$ the direct sum of copies of a module $X$.) 
We will also consider situation when no exponent $G$ is needed in (\ref{00}) and the density of $\biend{R}{U}$ in
$\biend{R}{U^*}$. In the proofs of these results we will use the maximal ring of 
quotients $\qr$. The use of this algebraic tool (which is not available, or at least not so effective, in functional analysis) 
nevertheless  forces  perhaps to restrictive assumptions for the validity of our results and the author hopes that experts in
algebra will be able to generalize the results considerably. With a 
different method, assuming that $C$ is contained in the center of $R$ and that $Z$ satisfies $\biend{C}{Z}=C$ (but without assuming that $TU=U$ or 
that $U$ is torsionless),  we will show here that the 
inclusion (\ref{000}) also holds for a large class of modules which includes all finitely related modules.

As a consequence of the results mentioned above it will follow that for many familiar algebras $R$ over a field the ring $\biend{R}{R^*}$
is the minimal possible, namely $R$. In an attempt to describe the ring $\biend{R}{R^*}$ in general, we consider the situation
when  $C$ is contained in the center of $R$ and observe that on $R^{**}:=\lhom{\com{C}}{R^*}{Z}$, where $\com{C}=\lend{C}{Z}$, there are 
two Arens associative products (familiar to specialists in Banach algebra theory). We describe
the ring $\biend{R}{R^*}$ in terms of these products  and characterize in purely algebraic terms when $R$ is
Arens regular (that is, when the two products
on $R^{**}$ coincide). If $R$ is Arens regular, $\biend{R}{R^*}$ turns out to maximal possible, namely $R^{**}$.  In Banach and locally convex algebra theory many interesting papers have been published concerning Arens products
(see \cite{P} and the references there) beginning with \cite{Ar}, \cite{Ar1}. The basic characterization of Arens regularity
in Banach algebras is in terms of weak compactness of certain maps. We will first prove here some basic results concerning Arens 
regularity within the purely algebraic context of duality theory of modules, then we specialize to algebras over fields. Since many 
familiar algebras $R$ over a field $\f$ have the property that $\biend{R}{R^*}=R$, while $\biend{R}{R^*}=R^{**}$ for Arens 
regular algebras (if the duality is defined in terms of $Z=\f$), the infinite dimensional Arens regular algebras over a field must be very special. We will show  that 
every such algebra contains an ideal $J$ of finite codimension with $J^2=0$. This may not be so for algebras over more general commutative
rings.

\section{Preliminaries}
 
{\em All rings here are assumed to be unital and the unit $1$ acts as the identity on all modules.} The category of all left (right)
modules over a ring $R$ is denoted by $\lm{R}$ ($\rmod{R}$, respectively).
For a faithful $X\in\rmod{R}$ we would like to consider $R$ as a subring of its biendomorphism ring 
$\biend{R}{X}=\lend{\lend{R}{X}}{X}$,
hence $X$ is regarded as a right module over $\biend{R}{X}$ and the ring multiplication in $\biend{R}{X}$ is the
reverse of the composition of maps. This suggests us to follow the convention of writing maps on the opposite side of 
scalars in general. Thus if $X\in\rmod{R}$, $x\in X$ and
$g\in\biend{R}{X}$, $xg$ means $g(x)$;  such an expression will be also written as $\inner{x}{g}$ if this
contributes to a greater clarity. 
  
Every $R$-module $U$ is contained in its {\em injective hull} $E(U)$ as a {\em large submodule} 
(that is, $U$ has nonzero intersection with all nonzero submodules of $E(U)$, see \cite{La}, \cite{Ro}). 

Regarding a ring $R$ as a right $R$-module, let $E=E(R)$, $H:=\lend{R}{E}$ and $\qr:=\lend{H}{E}$, the
{\em maximal right ring of quotients of $R$}. 

If $X\in\rmod{R}$ is injective and contains a copy of $R$ as a submodule, then $X$ contains also
a copy of $E=E(R)$, hence  $X$ 
is isomorphic to $E\oplus Y$ for a submodule $Y$ of $X$. Since each $\phi\in\biend{R}{X}$  commutes 
with the projection of $X$ onto $E$, it follows that $E\phi\subseteq E$ and therefore the restriction defines a ring 
homomorphism $\eta:\biend{R}{X}\to
\biend{R}{E}$.  
Since for each $x\in X$ the map $R\to X$, $r\mapsto xr$, can be extended to a homomorphism
$f:E\to X$ of $R$ modules, $X$ is the sum of the images of all such homomorphisms $f$ and consequently
$\eta$ must be injective \cite[14.1]{AF}. (Indeed, if $g\in\biend{R}{X}$ annihilates $E$, then $(fE)g=f(Eg)=0$ for all
$f\in\lhom{R}{E}{X}$, hence $Xg=0$.) Thus we may regard $\biend{R}{X}$ as a subring
in $\biend{R}{E}=\qr$. We need a more precise identification of this subring.

\begin{de}\label{deq} Let $Q_X$ be the set of all those $q\in\qr$ for which $f(1)=0$ implies that $f(q)=0$, whenever 
$f\in\lhom{R}{E}{X}$.
\end{de}

It can be verified that $Q_X$ is a subring of $\qr$ and clearly $R\subseteq Q_X$.

\begin{de}For each $x\in X$ and $q\in Q_X$ take $f\in\lhom{R}{E}{X}$ to be any extension of the map $\mu_x:R\to X$,
$\mu_x(r):=xr$, and define
$$xq:=f(q).$$
\end{de}

That this definition is unambiguous follows easily from the definition of $Q_X$. Part (i) of the following lemma
and its Corollary \ref{th2} are known \cite[p. 206]{S};  
short proofs are included here for convenience of readers. 

\begin{lemma}\label{le2}Suppose that $X\in\rmod{R}$ is injective and contains a copy of $R$ (as above). 

(i) Then the map $\varrho:Q_X\to\biend{R}{X}$, $\varrho(q)=\varrho_q$, where $\varrho_q$ is 
the right multiplication by $q$ on $X$, is an isomorphism of rings.

(ii) For each finite set $F\subseteq Q_X$, $\ann{X}{F^{-1}R}=0$, where $F^{-1}R:=\{r\in R:\, Fr\subseteq R\}$.
\end{lemma}

\begin{proof}(i) First observe that 
$\varrho_q$ is in $\biend{R}{X}$. Namely, given $x\in X$ and $q\in Q_X$,
let $f\in\lhom{R}{E}{X}$ satisfy $f(1)=x$, so that $xq=f(q)$. If
$g\in\lend{R}{X}$, then $gf\in\lhom{R}{E}{X}$ and $(gf)(1)=g(x)$, hence $g(x)q=(gf)(q)$, and we now conclude that
$g(xq)=g(f(q))=g(x)q$, so $\varrho_q\in\biend{R}{X}$.  Since  $\varrho_q|E$ is just right multiplication by $q$ on $E$, 
$\eta\varrho=1_{Q_X}$, where $\eta$ is the restriction map $\biend{R}{X}\to\biend{R}{E}
=\qr$. If we can show that $\im\eta\subseteq Q_X$, then $\varrho\eta$ is defined and from $\eta\varrho\eta=\eta$
and the injectivity of $\eta$ we will have that $\varrho\eta=1_{\biend{R}{X}}$, so
$\varrho$ and $\eta$ will be isomorphisms. To prove the inclusion $\im\eta\subseteq Q_X$, let $\phi\in\biend{R}{X}$ 
and note that $\phi|E$ is the right
multiplication by an element $q\in\qr$; we must show that $q\in Q_X$. So, let $f\in\lhom{R}{E}{X}$ satisfy
$f(1)=0$ and extend it to $f\in\lend{R}{X}$. Then $f(q)=f(1q)=f((1)\phi)=(f(1))\phi=0$, hence $q\in Q_X$.

(ii) Let $x\in X$, $F=\{q_1,\ldots,q_n\}$ and denote $q=(q_1,\ldots,q_n)$ and $D_q:=F^{-1}R$. If $xD_q=0$, then the map 
$qr+s\mapsto xr$ from the submodule
$qR+R^n$ of $E^n$ to $X$ is a well-defined homomorphism $f_0$ of right $R$-modules. (Namely, $qr+s=0$ implies that 
$qr=-s\in R^n$, hence $r\in D_q=F^{-1}R$ and therefore $xr=0$.) Hence $f_0$ can be extended to a map
$f\in\lhom{R}{E^n}{X}$. Note that $f$ is just an $n$-tuple of maps $f_j\in\lhom{R}{E}{X}$. Since $f(R^n)=0$,
$f_j(1)=0$ for all $j$. Hence $f_j(q_j)=0$, since $q_j\in Q_X$, and we conclude that $x=f(q)=\sumjn f_j(q_j)=0$.
\end{proof}

\begin{co}\label{th2}Let $X\in\rmod{R}$ be injective and faithful and $G\subseteq X$ a separating set for $R$
(that is, $Gr=0$ implies $r=0$ if $r\in R$). Then the ring $\biend{R}{X^G}$ is 
isomorphic to the subring $Q_{X^G}$ of $\qr$.
Thus if $X$ has a finite separating subset for $R$ then $\biend{R}{X}$ is isomorphic to a subring of $\qr$.
\end{co}

\begin{proof}The injection $R\to X^{G}$, $r\mapsto (gr)_{g\in G}$, extends to a monomorphism $E\to X^G$, hence we 
may  apply Lemma \ref{le2} to $X^G$ instead of $X$. 
It is well-known \cite[14.2]{AF} that for finite $G$ the rings $\biend{R}{X^G}$ and $\biend{R}{X}$ are 
isomorphic. 
\end{proof}

We recall that a submodule $Y$ of a right $R$-module $X$ is called {\em dense} in $X$ if for all $x,y\in X$ with $x\ne0$
there exists an $r\in R$ such that $xr\ne0$ and $yr\in Y$.

By definition the {\em trace ideal} $T(U)$ of a left $R$-module $U$ consists of all finite sums of elements of the form
$\inner{u}{f}$, where $u\in U$ and $f\in \lhom{R}{U}{R}$. It is well known  that $T:=T(U)$ is indeed
a two-sided ideal in $R$, which is idempotent (that is, $T^2=T$) if $TU=U$. For projective modules the following lemma is a part of \cite[1.16]{CF}; we will present
a short proof in order to see that the condition $U=TU$ suffices for our purposes. Modules satisfying $U=TU$ are
called {\em $T$-accessible} in \cite{Sa} and \cite{Z}. It is well-known that generators and projective modules are $T$-accessible, 
for more examples see  \cite[Section 3]{Z}.

\begin{lemma}\label{le11} If $U\in\lm{R}$ is faithful (so that we may regard $R$ as a subring of $B:=
\biend{R}{U}$), then $T$ is a left ideal in $B$. If in addition $U$ is $T$-accessible, 
then  $T$ is a dense right $R$-submodule of $B$ (hence $R$ is also dense in $B$).
\end{lemma}

\begin{proof}   For each $v\in U$ and $f\in
\lhom{R}{U}{R}$ denote by $f\diamond v$ the endomorphism of $U$ defined by $\inner{u}{f\diamond v}:=\inner{u}{f}v$ ($u\in U$).
Recall that $R$ is regarded as a subring of $B$ by identifying each $r\in R$ with the left multiplication $\lambda_r$ on $U$.
Then, since a biendomorphism $g\in B$ commutes with all endomorphisms (hence also with $f\diamond v$), for all 
$u,v\in U$ we have that $(g\inner{u}{f})v=(g\circ\lambda_{\inner{u}{f}})v=g(\inner{u}{f}v)=g(\inner{u}{f\diamond v})=
\inner{gu}{f\diamond v}=\inner{gu}{f}v$,
hence  
\begin{equation}\label{111}g\inner{u}{f}=\inner{gu}{f}\in T.\end{equation}
This implies that $T$ is a left ideal in $B$.

Now, since $BT\subseteq T$, to show that $T$ is dense in $B$, it suffices to show that for any nonzero
$b\in B$ there exists $t\in T$ such that $bt\ne0$. If there were no such
$t$, then $bT=0$, but then (since $U=TU$ by assumption) $b(U)=b(TU)=(bT)(U)=0$, hence
$b=0$. 
\end{proof}

Since $R$ is dense in $B$ by Lemma \ref{le11}, we may regard $B$ as a subring of $\qr$ by \cite[13.11]{La}.
That for a projective $R$-module $U$ the ring $B$ is contained
in $\qr$ was proved in \cite[1.16]{CF} and in \cite[2.3]{CRT}. We will need a more precise description of
$B$ for a more general class of modules provided by Lemma \ref{le12} below.
Note that since $R$ is dense in $\qr$ (as a right $R$-submodule) and since by Lemma \ref{le11} $T$ is dense in $R$ (assuming that  $TU=U$ 
and $U$ is faithful), $T$ is dense in $\qr$, hence the left annihilator of $T$ in $\qr$ is $0$. (To see this, apply the definition of density
to elements $q,1\in\qr$: if $q\ne0$ there exists $r\in R$ such that $qr\ne0$ and $1r\in T$, hence $qT\ne0$.)
Recall that a module $U\in\lm{R}$ is {\em torsionless}  if for each nonzero $u\in U$
there exists an $f\in\lhom{R}{U}{R}$ with $\inner{u}{f}\ne0$ (that is, $U$ can be embedded into $R^G$ for a sufficiently large
$G$). 

\begin{lemma}\label{le12}Let $U$ be a faithful left $R$ module, $T$ the trace ideal of $U$ in $R$ and $B=\biend{R}{U}$.
Suppose that $U$ is torsionless and $T$-accessible. Then $B=\{q\in\qr:\, qT\subseteq T\}$. Moreover, $D_qU=U$ for each 
$q\in B$, where $D_q=q^{-1}R$.
\end{lemma}

\begin{proof}By Lemma \ref{le11} $qT\subseteq T$ for each $q\in B$. Assume now that $q\in\qr$ and $qT\subseteq T$.
Since $TU=U$, each $u\in U$ can be expressed as a finite sum $u=\sum \inner{u_i}{f_i}v_i$, where $u_i,v_i\in U$ and $f_i\in
\lhom{R}{U}{R}$ may depend on $u$. Define a map $g:U\to U$ by $g(u)=\sum (q\inner{u_i}{f_i})v_i$ (which makes sense since 
$q\inner{u_i}{f_i}\in R$ because $\inner{u_i}{f_i}\in T$ and $qT\subseteq T$). To show that $g$ is well-defined,
it suffices to prove that $\sum \inner{u_i}{f_i}v_i=0$ implies that $\sum (q\inner{u_i}{f_i})v_i=0$. For this, since
$U$ is torsionless, it suffices to observe that
for each $f\in\lhom{R}{U}{R}$ we have
\begin{align*}\inner{\sum (q\inner{u_i}{f_i})v_i}{f}=\sum (q\inner{u_i}{f_i})\inner{v_i}{f}
=q\sum \inner{u_i}{f_i}\inner{v_i}{f}\\
=q\inner{\sum \inner{u_i}{f_i}v_i}{f}=0.
\end{align*}
It is straightforward to verify that $g$ commutes with all $R$-endomorphisms of $U$ so that $g\in B$. To show that
$g$, when regarded as an element of $\qr$, is just $q$, it suffices to show that $qt=gt$ for each $t\in T$ (since
$\ann{\qr}{T}=0$ as observed above). If $t$ is of the form $t=\inner{u}{f}$ ($u\in U$, $f\in\lhom{R}{U}{R}$), then by 
(\ref{111}) $gt=\inner{g(u)}{f}$, hence $(gt)v=\inner{g(u)}{f}v$ for all $v\in U$. But by the definition of $g$ we have 
$g(tv)=g(\inner{u}{f}v)=(q\inner{u}{f})v=(qt)v$, hence $(gt)v=(qt)v$. Since $U$ is faithful, this implies that $gt=qt$. 
Since each
element of $T$ is a sum of elements of the form $\inner{u}{f}$, this concludes the proof that $g=q$ in $\qr$.

If $q\in B$, then as proved above $qT\subseteq T$, hence $T\subseteq D_q$ and it follows that
$D_qU\supseteq TU=U$.
\end{proof}

Recall that a module $U\in\lm{R}$ is a {\em generator} if for some $n$ there is an epimorphism $U^n\to R$.

\begin{co}\label{co121}If $U\in\lm{R}$ is a  torsionless generator, then $\biend{R}{U}=R$.
\end{co}

\begin{proof}Since $U$ is a generator, $T=R$ \cite[18.8]{La}, hence $TU=U$. We have observed in the proof of Lemma
\ref{le12} that $D_q\supseteq T$ for each $q\in B$, hence $D_q=R$. Thus $q=q1\in R$ (since
$1\in D_q$).
\end{proof}

A module $Z\in\rmod{C}$ is called a {\em cogenerator}  
if for every $U\in\rmod{C}$ and every nonzero $u\in U$ there exists an $f\in\lhom{C}{U}{Z}$
such that $f(u)\ne0$.

{\em Throughout the paper  $Z\in\rmod{C}$ will always be an injective cogenerator,
$\com{C}$ will be the ring $\com{C}:=\lend{C}{Z}$ and $\bcom{C}:=\lend{\com{C}}{Z}$.} So $Z$ is a 
$\com{C},\bcom{C}$-bimodule and $C$ is regarded as a subring of $\bcom{C}$ by identifying each $c\in C$ with
the corresponding right multiplication on $Z$. (Note that $Z$ is faithful as a $C$-module since it is a cogenerator.)
If $C$ is abelian then $C\subseteq\bcom{C}\subseteq\com{C}$.

\begin{de}Given a fixed injective cogenerator $Z\in\rmod{C}$, for any  $U\in\rmod{C}$ the {\em dual} is defined by  
$U^*:=\lhom{C}{U}{Z}$. Then $U^*$ is  a left $\com{C}$-module by $t\rho:=t\circ\rho$, 
where $t\in \com{C}$ and $\rho\in U^*$. For a left module $W\in\lm{\com{C}}$ the dual is defined by $W^*:=\lhom{\com{C}}{W}{Z}$. 
In particular for $U\in\rmod{C}$ we have the bidual $U^{**}=(U^*)^*$. Then $U^{**}$ is a right $\bcom{C}$-module
by $\theta s:=s\circ\theta$, where $\theta\in U^{**}$, $s\in\bcom{C}$(recall that the multiplication in $\bcom{C}$
is the reverse of the composition of maps).
Then we can define $U^{***}:=\lhom{\bcom{C}}{U^{**}}{Z}$
(homomorphisms of right $\bcom{C}$-bimodules).
\end{de}

We will need the second and the third dual only in the case $C$ is commutative, contained in the center of $R$.

In the special case when $Z=C$ is a field, $U^{**}$ is just the familiar bidual of $U$.  {\em To avoid a possible ambiguity, 
in this paper  $U$ will always denote a right $C$-module and $U^*$, $U^{**}$ and $U^{***}$ its consecutive duals
as defined in the above definition.}

The natural homomorphism 
$\iota:U\to U^{**}$ of  
$C$-modules is injective since $Z$ is a cogenerator in $\rmod{C}$.
We will also use the notation 
$\inner{\rho}{u}:=\rho(u)\ \ (u\in U,\ \rho\in U^*).$

\begin{re}\label{re0}If $U_0$ is a $C$-submodule of a module $U\in\rmod{C}$, then the restriction map
$U^*\to U_0^*$, $\rho\mapsto\rho|U_0$ is surjective (by the injectivity of $Z$), and its kernel is $U_0^{\perp}$,
the annihilator of $U_0$ in $U^*$, hence $U_0^*=U^*/U_0^{\perp}$ (where the equality means in fact the natural isomorphism).
Further, if $W_0$ is a submodule of a module $W\in\lm{\com{C}}$, then clearly $(W/W_0)^*$ is naturally isomorphic to
$W_0^{\perp}$, the annihilator of $W_0$ in $W^*$, thus $(W/W_0)^*=W_0^{\perp}$. In particular, taking  
$W=U^*$ and $W_0=U_0^{\perp}$, we have 

$$U_0^{**}=(U^*/U_0^{\perp})^*=U_0^{\perp\perp}.$$
\end{re}

If $U,V\in\rmod{C}$ and  $a\in\lhom{C}{U}{V}$, the {\em adjoint} $a^*\in\lhom{\com{C}}{V^*}{U^*}$ is defined by 
$\inner{\rho}{a^*}=\rho\circ a$ ($\rho\in V^*$).
Similarly, for $b\in\lhom{\com{C}}{V^*}{U^*}$  we may define $b^*\in\lhom{\bcom{C}}{U^{**}}{V^{**}}$ by $\inner{b^*}{\theta}=
\theta\circ b$. For $a\in\lhom{C}{U}{V}$ we set $a^{**}=(a^*)^*$. Thus 
$a^{**}\in\lhom{\bcom{C}}{U^{**}}{V^{**}}$. 

\begin{lemma}\label{le301}(i) Let $a\in\lhom{C}{U}{V}$. Then $\ker a^{**}=(\im a^*)^{\perp}$
(the annihilator of $\im a^*$ in $U^{**}$). Since $Z$ is injective in $\rmod{C}$, we also have that 
$\im a^*=(\ker a)^{\perp}$, hence 
$\ker a^{**}=((\ker a)^{\perp})^{\perp}.$ Thus
$$\ker a^{**}=(\ker a)^{**}.$$

(ii) If $Z$ is injective also as a left $\com{C}$-module, then 
$$a(U)^{**}=a(U)^{\perp\perp}=a^{**}(U^{**}).$$

(iii) Since $Z$ is a cogenerator in $\rmod{C}$ and $\com{C}=\lend{C}{Z}$,  
$U$ is dense in $U^{**}$
in the following sense: for each $\theta\in U^{**}$ and each finite subset 
$\{\rho_1,\ldots,\rho_n\}$ of $U^*$ there exists  $u\in U$ such that $\rho_j\theta:=\rho_ju$
for all $j=1,\ldots n$. 
\end{lemma}

\begin{proof} The proof of (i) is a routine verification (where Remark \ref{re0} is used to prove the last formula in (i)).

(ii) The first equality is just a special case of the last formula in Remark \ref{re0}. For the second equality, first note that
$a^{**}(U^{**})=(\ker a^*)^{\perp}$. Indeed, the inclusion $\subseteq$ is obvious and for the reverse inclusion, given $\theta\in(\ker a^*)^{\perp}$, 
we need to find $\sigma\in U^{**}$ such that $\sigma\circ a^*=\theta$. Since $\theta$ annihilates $\ker a^*$, we may 
define $\sigma$ on the
range of $a^*$ by $\inner{\inner{\rho}{a^*}}{\sigma}:=\inner{\rho}{\theta}$ ($\rho\in V^*$) and then extend it by 
$\com{C}$-injectivity of $Z$
to $\sigma\in\lhom{\com{C}}{U^*}{Z}=U^{**}$. Finally, using the obvious identity $\ker a^*=(aU)^{\perp}$, we have
$a^{**}(U^{**})=(\ker a^*)^{\perp}=(aU)^{\perp\perp}$.

(iii) Let $Y:=\{(\rho_1x,\ldots,\rho_nx):\, x\in U\}$,
and  $z=(\rho_1\theta,\ldots,\rho_n\theta)$. We must show that $z\in Y$.
If $z\notin Y$, then  (since $Z$ is a cogenerator in $\rmod{C}$) there exists $t\in\lhom{C}{Z^n}{Z}$ 
such that $tY=0$ and $tz\ne0$.
But the fact that $t$ is just a $n$-tuple of maps $t_j\in\com{C}$ leads to
$\sum_{j=1}^nt_j\rho_j=0$ and (since $t_j\in\com{C}$ and $\theta$ is a homomorphism of $\com{C}$-modules) to
$(\sum_{j=1}^nt_j\rho_j)\theta=tz\ne0$, a contradiction.)
\end{proof}

A module $U\in\lm{R}$ is {\em linearly compact} if for any family of submodules $U_i$ of $U$ and elements $u_i\in U$
the system of relations $x-u_i\in U_i$ has a solution $x$ provided that every finite subsystem has a solution.

{\em By a $C$-algebra we mean a ring $R$ containing $C$ as a subring in its center so that $C$ and $R$ have  
the same identity $1$. (In particular $C$ is commutative.)}

If $R$ is a $C$-algebra and $U\in\lm{R}$, then $U$ is also a $C$-module, hence $U^*=\lhom{C}{U}{Z}$ 
is defined. If $V\in\lm{R}$, then each
$f\in\lhom{R}{U}{V}$ is also a homomorphism of  $C$-modules, hence so is its adjoint $f^*:V^*\to U^*$. 
Note also that
$U^*$ (and $V^*$) is a right $R$-module by $\rho r:=\rho\circ \lambda_r$, where $\lambda_r$ is the left multiplication
by $r\in R$ on $U$. Moreover,  $U^*$ is a $\com{C},R$-bimodule. Consequently 
$U^{**}$ is a left 
$R$-module by $(r\theta)(\rho):=\theta(\rho r)$ ($\theta\in U^{**}$, $\rho\in U^*$, $r\in R$), and the natural map 
$U\to U^{**}$ is a homomorphism of
left $R$-modules.  The adjoint $f^*$ of each homomorphism $f\in\lhom{R}{U}{V}$ is a homomorphism of right $R$-modules and
also a homomorphism of  $\com{C}$-modules. Since $C$ is contained in the center of $R$, elements of $C$ act as $R$-endomorphisms of $U$, 
hence $\biend{R}{U}\subseteq\lend{C}{U}$ and so for each $b\in\biend{R}{U}$ the adjoint $b^*$ is defined.

\section{Biendomorphisms of torsionless  modules}

Recall that if $U$ is an $R,C$-bimodule which is flat as a left $R$-module and $Z\in\rmod{C}$ is injective,  then
$U^*=\lhom{C}{U}{Z}$ is injective in $\rmod{R}$ \cite[3.5]{La}, where the right action of $R$ on $U^*$ is defined by
$\inner{u}{\rho r}=\inner{ru}{\rho}$ ($u\in U$, $r\in R$, $\rho\in U^*$). 

Observe that a subset $G\subseteq U^*$ is {\em separating for $R$} (that is, the map $R\mapsto (U^*)^G$, $r\mapsto(\rho r)_{\rho\in G}$, is injective)
iff the inclusion $rU\subseteq\bigcap_{\rho\in G}\ker\rho$ for $r\in R$ implies that $r=0$. Note also that if $U$ is 
faithful as an $R$-module, so is $U^*=\lhom{C}{U}{Z}\in\rmod{R}$ since $Z$ is a cogenerator in $\rmod{C}$, thus 
$U^*$ has a separating subset $G$ for $R$ (for example, $G=U^*$).

\begin{theorem}\label{th3}Let $U$ be an $R,C$-bimodule which is faithful as a left $R$-module,
$Z$ an injective cogenerator in $\rmod{C}$ and $U^*=\lhom{C}{U}{Z}$. 

(i) If $U\in\lm{R}$ is torsionless and $T$-accessible (that is, $TU=U$, where $T$ is the trace ideal of $U$ in $R$), 
then 
\begin{equation}\biend{R}{U}\subseteq\biend{R}{U^*}.\end{equation}

(ii) If $U\in\lm{R}$ is flat and $G\subseteq U^*$ is a separating subset for $R$, then  
\begin{equation}\label{8}\biend{R}{(U^*)^G}\subseteq\biend{R}{U^{(G)}}\ \cong(\biend{R}{U}.\end{equation}
Thus if   $U^*$ has a finite separating subset for $R$ (which is always the case if $R$ is right artinian), then
\begin{equation}\label{2}\biend{R}{U^*}\subseteq\biend{R}{U}.\end{equation}
\end{theorem}

\begin{proof} First note that since each $c\in C$ acts as an $R$-endomorphism of $U$, $c$ commutes with every
$q\in B:=\biend{R}{U}$, hence $q\in\lend{C}{U}$. Therefore $q^*$ is defined and $\rho\circ q\in U^*$ for each
$\rho\in U^*$. 

(i) If $U$ is torsionless and $TU=U$, then we will show that $q^*\in\biend{R}{U^*}$ for each $q\in B$; in other words, 
that $q^*\circ g=g\circ q^*$ for all
$g\in\lend{R}{U^*}$, which means that $g(\rho)\circ q=g(\rho\circ q)$ for all $\rho\in U^*$. By Lemma \ref{le12} 
$B$ is a subring of $\qr$. For each $q\in B$ let $D_q:=q^{-1}R=\{r\in R:\, qr\in R\}$.
As usual, we regard $R$ as a subring in $B$ by identifying each $r\in R$ with the left multiplication $\lambda_r$ on $U$.
Then $(\rho\circ q)r=\rho\circ(qr)$ for all $q\in B$, $\rho\in U^*$ and $r\in R$ (this is just the associativity
$(\rho\circ q)\circ \lambda_r=\rho\circ(q\circ \lambda_r)$). Applying this to $g(\rho)$
instead of $\rho$, we obtain that $(g(\rho)\circ q)r=g(\rho)\circ(qr)$. If $r\in D_q$, then $s:=qr\in R$ and,
since $g$ is an $R$-endomorphism,  we have now $(g(\rho)\circ q-g(\rho\circ q))r=g(\rho)\circ(qr)-g((\rho\circ q)r)
=g(\rho)\circ(qr)-g(\rho\circ(qr))
=g(\rho) s-g(\rho s)=0$. 
Thus $\omega:=g(\rho)\circ q-g(\rho\circ q)\in U^*$ satisfies $\omega D_q=0$, that is $\omega(D_qU)=0$. Since
$D_qU=U$ by Lemma \ref{le12}, $\omega=0$, hence $g(\rho)\circ q=g(\rho\circ q)$, which means that $q^*\in\biend{R}{U^*}$. 

(ii) Denote  $W:=U^{(G)}$ and $V:=(U^*)^G=W^*$.  By
Corollary \ref{th2} $\biend{R}{V}$ can be identified with the subring $Q:=Q_V$ of $\qr$. So, to prove 
(\ref{8}), we will define a left $Q$-module structure on $W$ (and on $U$) so that for each $q\in Q$ the right 
multiplication $\varrho_q$ by $q$ on $V$ will be just the adjoint operator of the left multiplication
$\lambda_q$ by $q$ on $W$. Assuming this, and noting that  for each $f\in\lend{R}{W}$ the map $f^*$ is in $\lend{R}{V}$,
hence $f^*$ commutes with $\lambda_q^*=\varrho_q\in\biend{R}{V}$, we will have that $\lambda_qf=f\lambda_q$, which 
means that $\lambda_q\in\biend{R}{W}$. Since the map $q\mapsto \lambda_q$
is injective (for $\lambda_q=0$ only if $\varrho_q=\lambda_q^*=0$,
hence $q=0$), this will show that $Q$ can be regarded as a subset of $\biend{R}{W}$ and consequently (\ref{8}) will follow.

Note that by 
Lemma \ref{le2}(ii) for each $q\in Q$ the annihilator in $V$ of  $D_q:=q^{-1}V$ is $0$. This means that 
$D_qW=W$, for otherwise $\inner{D_qW}{v}=0$ for some
nonzero $v\in V$ (since $Z$ is a cogenerator in $\rmod{C}$), which would mean that $\inner{W}{vD_q}=0$, hence that 
$v\in\ann{V}{D_q}$ and therefore $\ann{V}{D_q}$ would not be $0$.

Since $W=D_qW$, we may now  define a $Q$-module structure on $W$ by
\begin{equation}\label{41}q(\sumjn a_jw_j):=\sumjn(qa_j)w_j\ \ (a_j\in D_q,\ w_j\in W).\end{equation}
To show that this is well defined, assume that $\sumjn a_jw_j=0$.
Then the vector $w:=(w_1,\ldots,w_n)\in W^n$ is in the kernel of the operator
$a:W^n\to W$ defined by $a(y_1,\ldots,y_n)=\sumjn a_jy_j$. But $\ker a=(\im a^*)_{\perp}$, so $w\in(\im a^*)_{\perp}$. 
Since $a_j\in D_q$, $b_j:=qa_j\in R$, so the right multiplication by $b_j$ on $V$
is just the adjoint operator of the left multiplication by $b_j$ on $W$ and similarly for $a_j$. So we consider now the
operators $a:=(a_1,\ldots,a_n)$ and $b:=(b_1,\ldots,b_n)$ from $W^n$ to $W$ with the adjoints
$a^*=[a_1^*,\ldots,a_n^*]:V\to V^n$ and $b^*=[b_1,\ldots,b_n]:V\to V^n$.
By definition of adjoints we have that $va:=(va_1,\ldots,va_n)=(\inner{v}{a_1^*},\ldots,\inner{v}{a_n^*})=\inner{v}{a^*}$ 
for all $v\in V$,
hence $Va=\im {a^*}$ and similarly $Vb=\im{b^*}$. Since
$b=qa$, it follows that
$$\im b^*=Vb=V(qa)=(Vq)a=\inner{Vq}{a^*}\subseteq\im a^*.$$
Since $w\in(\im a^*)_{\perp}$, it follows now that $w\in(\im b^*)_{\perp}=\ker b$,
which means that $\sumjn(qa_j)w_j=bw=0$. This proves that (\ref{41}) is well defined. It can be verified
routinely (using Lemma \ref{le2}(ii)) that in this way $W$ becomes a left $Q$ module such that $V$ is its dual module. 
Further, it follows from
(\ref{41}) that $q$ preserves all the summands $U$ in the decomposition $W=U^{(G)}$,
hence $U$ is also a $Q$-module and $W=U^{(G)}$ as left $Q$-modules.

If $U^*$ has a finite separating set $G$ (\ref{2}) follows from (\ref{8}), since for a finite $G$ we have $\biend{R}{U}\cong\biend{R}{U^G}$ and 
$\biend{R}{U^*}\cong\biend{R}{(U^*)^G}$. If $U^*$ does not have any finite separating subset for $R$, then for any
$\rho_1\in U^*$ the annihilator $\ann{R}{\rho_1}$ is nonzero (since $\{\rho_1\}$ is not separating for $R$). Since $U^*$ is 
faithful there exists $\rho_2\in U^*$ such that $\rho_2\ann{R}{\rho_1}\ne0$, hence $\ann{R}{\rho_1,\rho_2}\subset
\ann{R}{\rho_1}$ (where the inclusion is strict). Continuing in this way, we can find a strictly decreasing sequence
of right ideals $\ann{R}{\rho_1}\supset\ann{R}{\rho_1,\rho_2}\supset\ldots$. Thus, if $R$ is right artinian, $U^*$ must
have a finite separating subset for $R$.
\end{proof}

The set $G$ in Theorem \ref{th3} is not redundant, that is, the inclusion (\ref{2})
does not always hold even if $U$ is projective and faithful. For example, if $U$ is an infinite
dimensional vector space over a field $C=\f$, $Z=\f$ and $R=\lend{\f}{U}$, then $\biend{R}{U^*}=\lend{\f}{U^*}$ contains also
operators which are not adjoint to any operator in $R=\biend{R}{U}$.

If $R$ is a $C$-algebra, the module $R^*$ plays a special role. We will now observe that if $R^*$ has a finite separating subset for $R$,
then the same holds for many modules $U^*$. Suppose that besides $R^*$ also $U$ has a finite separating subset for $R$.
This means that there exist monomorphisms $\mu:R\to U^m$  and $\nu:R\to (R^*)^n$ of $R$-modules 
(left and right, respectively) for some
$m,n\in\bn$. Then the direct sum of $n$ copies of $\mu^*$ is an epimorphism $(U^*)^{mn}\to (R^*)^n$ of right $R$-modules, 
and we may
lift $\nu(1)$ to an element $\omega$ of $(U^*)^{mn}$, which defines a monomorphism $R\to (U^*)^{mn}$, $r\mapsto\omega r$
of right $R$-modules.
So in this case $U^*$ has a finite separating subset for $R$ and therefore (\ref{2}) holds.
In particular, if $R$ is a domain, then every nonzero element in a flat $R$-module $U$ is separating since flat modules are torsion-free
\cite[4.18]{La}, so in this case we have the above monomorphism $\mu$ (with $m=1$).
Note also that a generator $U\in\lm{R}$ has a finite separating subset. (Namely, a generator means that for some $m\in\bn$ there 
exists an epimorphism $U^m\to R$,  hence also a monomorphism from $R$ to $U^m$ since $1\in R$ can be lifted to an element 
of $U^m$.) 
If $U$ is torsionless and $TU=U$ (for example, $U=R$), then equality holds in (\ref{2}) by Theorem  \ref{th3}. 
Thus we have the following corollary.

\begin{co}\label{co31}Let $R$ be a $C$-algebra. Suppose that $U\in\lm{R}$ is flat and $R^*$ has a finite separating subset for $R$. 

(i) If $U$ contains a
finite separating subset for $R$ (in particular, if $R$ is a domain or, if $U$ is a generator in $\lm{R}$), then the inclusion 
(\ref{2}) holds. 

(ii)  $\biend{R}{R^*}=R$. 
\end{co}

Suppose that $Z=C$ and $C$ is a field $\f$. It turns out that the class of $\f$-algebras $R$ such that $R^*$ has a 
separating functional $\omega$ for $R$ (so that Corollary \ref{co31} applies) includes  all primitive 
algebras with nonzero socle,  von Neumann algebras (where $\f=\mathbb{C}$)
and various function algebras, such as polynomial algebras $R=\f[x_1,\ldots,x_n]$ over 
fields $\f\in\{\mathbb{C},\mathbb{R}\}$. 
On the other hand, if $R=\f\oplus R_0$,  where $R_0$ is an infinite dimensional vector space over $\f$,  and the multiplication
on $R$ is defined so that $R_0R_0=0$ and $1\in\f$ acts as the identity in $R$, then it can be proved that $R^*$ does not 
admit any finite separating subset for $R$. (In fact, it turns out that in this example 
$\biend{R}{R^*}\cong R^{**}\ne R$. We will consider a more general situation in the last section.)

\medskip
{\em Problem.} Is the inclusion ({\ref{2}) true for all $R$-modules if $R$ is, say, a commutative noetherian 
algebra over a field $\f$ and $Z=\f$?

\section{Biendomorphisms of certain modules admitting torsion}

In this section we will prove the inclusion 
$\biend{R}{U}\subseteq\biend{R}{U^*}$
for a large class of (not necessarily flat) modules over $C$-algebras.

Recall that a  {\em finitely related} module is a quotient of a free
module by a finitely generated  submodule.

\begin{de}A left $R$-module $U$ has the {\em FRCP (finitely related complementation property)}
if $U$ is a direct summand in a left $R$-module $W$ with the property 
that each finite subset of $W$ is contained in a finitely related direct summand  
of $W$. 
\end{de}

All finitely related and all projective modules have the FRCP. It can be proved (but not needed here) that
if $R$ is left and right noetherian then $R^{\bn}$ 
has the FRCP.  (Note that $R^{\bn}$ is not projective if, for example, $R=\bz$ \cite[p. 22]{La}.)

To each $g\in\lhom{R}{V^*}{U^*}$ we can associate  the formal adjoint map $g^*|U$,   but
to assure that its range is in $V^{**}$, we need that $g$ is a homomorphism of left $\com{C}$-modules (this is automatic if 
$\com{C}=C$, since
$C\subseteq R$). So, let $\ltr$ denote the set of all $\com{C},R$-bimodule homomorphisms from $V^*$ to $U^*$. Then 
each $g\in\ltr$ has an adjoint $g^*\in\lhom{R}{U^{**}}{V^{**}}$. It is convenient  that the natural homomorphism 
of $C$-modules
\begin{equation}\label{thom}\ltr\to\lhom{R}{U}{V^{**}},\ \ g\mapsto g^*|U\end{equation}
turns out to be an isomorphism. Indeed, the inverse map sends  $h\in \lhom{R}{U}{V^{**}}$ to $h^*|V^*$, where $V^*$ is regarded
as a submodule of $V^{***}$. 
We  need the following generalization of \cite[Theorem 2.2]{M}.

\begin{theorem}\label{th1}Let $R$ be a $C$-algebra,  $U, V\in\lm{R}$, let $Z\in\rmod{C}$ be an injective cogenerator 
and let $U^*=\lhom{C}{U}{Z}$ and $V^*=\lhom{C}{V}{Z}$. If $U$ has the FRCP 
then each $g\in\ltr$ can be approximated by adjoints of elements of $\lhom{R}{U}{V}$ in the following sense:
for all finite subsets $G$ of $U$ and $H$ of $V^*$ there exists $f\in\lhom{R}{U}{V}$ such that
$$\inner{u}{g\rho}=\inner{uf}{\rho}\ \ \mbox{for all}\ \rho\in H\ \mbox{and}\ u\in G.$$
\end{theorem}

\begin{proof}First consider the case when $U$ is finitely related, that is, $U=R^{(I)}/A$ for some  $I$ and a finitely generated submodule $A$ of 
$R^{(I)}$. (Here $R^{(I)}\subseteq R^{I}$ consists of elements which have only finitely many nonzero components.)
Let $\{r_1,\ldots,r_m\}$ be a set of generators of $A$. Using the standard basis $(e_i)_{i\in I}$ of $R^{(I)}$, $\lhom{R}{R^{(I)}}{V}$
can be naturally identified with $V^I$ (by identifying each $f\in\lhom{R}{R^{(I)}}{V}$ with $(e_if)_{i\in I}\in V^I$) and consequently  
$$\lhom{R}{U}{V}=\{f\in\lhom{R}{R^{(I)}}{V}:\, Af=0\}=\ann{V^I}{A}.$$
Similarly, using the natural isomorphism (\ref{thom}) 
we have that
$$\ltr=\ann{(V^{**})^I}{A}.$$
The space $(V^{**})^I$ is the dual of $(V^*)^{(I)}$ and it can be verified that under the above identifications the
theorem translates to the following statement: {\em given $\theta=(\theta_j)\in(V^{**})^I$ which annihilates $A$, for each finite subset $H_0$ of 
$(V^*)^{(I)}$ there exists $v=(v_j)\in V^I$ annihilating $A$ such that 
$\inner{v}{\rho}=\inner{\theta}{\rho}$ for all $\rho\in H_0$.} 
Denoting by $r_{i,j}$ the components of the generators $r_i$ of $A$, an element $v=(v_j)\in V^I$ (respectively
an element $\theta=(\theta_j)\in (V^{**})^I$) is in
$\ann{V^I}{A}$ (resp. in $\ann{(V^{**})^I}{A}$) if and only if
$$\sum_{j\in I}r_{i,j}v_j=0\ \ (\mbox{resp.} \sum_{j\in I}r_{i,j}\theta_j=0)\ \ \mbox{for all}\ i=1,\ldots m.$$
Since all $\rho\in H_0$ and all $r_i$ have only finitely many non-zero components, there exists a finite subset $n$ 
of $I$ such that $r_{i,j}=0$ and $\rho_j=0$ for all $i\in\{1,\ldots,m\}$ and all $\rho\in H_0$, if $j\in I\setminus n$. Let $r:V^n\to V^m$ be the homomorphism of $C$-modules
defined by $(r((v_j)))_i=\sum_{j\in n}r_{i,j}v_j$. Then $r^{**}:(V^{**})^n\to (V^{**})^m$ is given by
$(r^{**}((\theta_j)))_i=\sum_{j\in n}{r_{i,j}^{**}}\theta_j$ and it follows that
\begin{equation}\label{02}\ann{V^I}{A}=\ker r\times V^{I\setminus n}\ \ \mbox{and}\ \ 
\ann{(V^{**})^I}{A}=\ker r^{**}\times(V^{**})^{I\setminus n}.\end{equation}
By Lemma \ref{le301}(iii) $V^{k}$ is dense in $(V^{k})^{**}=(V^{**})^k$ for each finite $k$ and then by Lemma 
\ref{le301}(i) $\ker r$
is dense in $\ker r^{**}$. Since all the components of elements of $H_0\subseteq (V^*)^{(I)}$ are zero outside
of $n$, we see now from (\ref{02}) that $\ann{V^I}{A}$ is dense in $\ann{(V^{**})^{I}}{A}$ in the appropriate sense, which proves the
theorem for finitely related modules.

In general, let $U\in \lm{R}$ be a direct summand in $W\in\lm{R}$, where each finite subset of $W$ is
contained in a finitely related complemented $R$-submodule of $W$, and let $p:W\to U$ be a
projection, so that $p^*:U^*\to W^*$ is a monomorphism of right $R$-modules and also a homomorphism of left 
$\com{C}$-modules. Given
finite subsets $G\subseteq U$ and $H\subseteq V^*$ and $g\in\ltr$,
we consider the composition $i^*\circ p^*\circ g:V^*\to U_G^*$, where $U_G$ is a complemented
finitely related submodule of $W$ containing $G$ and $i:U_G\to W$ is the inclusion.
By what we have already proved there exists $f_0\in\lhom{R}{U_G}{V}$ such that
$\inner{u}{(i^*\circ p^*\circ g)\rho}=\inner{uf_0}{\rho}$ for all $\rho\in H$ and $u\in G$.
Let $f:=f_0\circ(q|U)$, where $q:W\to U_G$ is a projection. Then $f\in\lhom{R}{U}{V}$ and 
$\inner{u}{g\rho}=\inner{uip}{g\rho}
=\inner{u}{i^*p^*g\rho}=\inner{uf_0}{\rho}=\inner{uf}{\rho}$ for all $u\in G$
and $\rho\in H$.
\end{proof}

Assume that $U$ has the FRCP.  If $b\in\biend{R}{U}$, then $fb=bf$ for all $f\in\lend{R}{U}$, hence  $b^*f^*=f^*b^*$. 
Since  every $g\in\etr$ can be approximated by maps $f^*$ by Theorem \ref{th1},
it follows that $b^*g=gb^*$. Hence $b^*\in\lend{\etr}{U^*}$. By associativity of composition of maps we also have that 
$b^*\in\lend{\com{C}}{U^*}$, hence $b\in \lend{\etr}{U^*}\cap\lend{\com{C}}{U^*}$ and
we can state the following corollary.

\begin{co}\label{co1}If in Theorem \ref{th1} $V=U$  then 
\begin{equation}\label{01}\biend{R}{U}\subseteq\lend{\etr}{U^*}\cap\lend{\com{C}}{U^*},\end{equation}
where $\biend{R}{U}=\{g^*:\, g\in\biend{R}{U}\}$. In particular, if $\com{C}=C$, then
\begin{equation}\label{1}\biend{R}{U}\subseteq\biend{R}{U^*},
\end{equation} 
\end{co}

\begin{re}If $U$ is faithful and flat over $R$ with $TU=U$ then by Theorem \ref{th3} (\ref{1}) holds even if $\com{C}\ne C$. 
The author does not know 
if  (\ref{1}) can fail when the condition $\com{C}=C$ is omitted.
(Note that since $Z$ is an injective cogenerator in $\rmod{C}$ and $C$ is commutative, the assumption $\com{C}=C$ 
implies that $C$ admits a Morita duality by  \cite[19.43]{La}, hence  by \cite{A} $C$ must  be 
linearly compact in $\rmod{C}$.) 
\end{re}

\section{Density}

In general, for a faithful flat $U\in\lm{R}$ the inclusion $\tilde{B}=\biend{R}{U^*}\subseteq\biend{R}{U}=B$
does not hold (as we have noted immediately after Theorem \ref{th3}) and it is more natural to ask if at least 
$\tilde{B}$ is contained is an appropriate closure of 
$B$. 

If $U\in\lm{R}$ is $T$-accessible, then for many rings $R$ the condition
that $u\in Ru$ for all $u\in U^n$ and all $n\in\mathbb{N}$ (used in the Proposition \ref{pr400} below) is automatically
satisfied. We will see in a moment that this is so if $R$ has the following property:  
for each idempotent two-sided ideal $J$ of $R$ (that is, $J^2=J)$ and each finite subset
$F$ of $J$ there exists an element $e\in J$ such that $er=r$ for all $r\in F$.
(This property holds, for example, if all idempotent two-sided ideals in $R$ are generated by idempotents, which includes
all commutative noetherian rings by \cite[2.43]{La}.)
Namely, in this case, given $u=(u_1,\ldots,u_n)\in U^n$, since $TU=U$, each $u_i$ is a finite sum $\sum_jt_{i,j}u_{i,j}$,
where $t_{i,j}\in T$ and $u_{i,j}\in U$. If  $e\in T$ is such that $et_{i,j}=t_{i,j}$ for all $i,j$, then $eu=u$,
hence $u\in Tu$. The first part of the following proposition, or at least a variation of it, is known \cite[1.3]{F}, but we will present a 
very short direct proof. The hypothesis that $u\in Tu$ can be replaced by the assumption that $U$ is a $\Sigma$-self generator in the sense 
of \cite{Z} without essentially changing the proof also in the second part of the proposition.

\begin{pr}\label{pr400} Let $U$ be an $R,C$-bimodule which is faithful as a left $R$-module and let $T$ be the trace ideal (in $R$) 
of $U$ as a left $R$-module. Suppose that $u\in Tu$
for each $u\in U^n$ and each $n\in\mathbb{N}$.   

(i) Then $R$ is dense in $B:=\biend{R}{U}$ in the sense that
for each $b\in B$ and each finite subset $G\subseteq U$ there exists an $r\in R$ such that $ru=bu$ for all $u\in G$.

(ii) Suppose that  $C$ is contained in the center of $R$ and $cu=uc$ for all $u\in U$ and $c\in C$. If $U$ is torsionless and flat as an $R$-module and if $\lend{C}{Z}=C$, then $R$ is weakly dense in $\tilde{B}:=\biend{R}{U^*}$ in the sense that for each
$s\in\tilde{B}$ and each finite subset $H$ of $U\times U^*$ there exists $r\in R$ such that $\inner{u}{\rho r}=\inner{u}{\rho s}$
for all $(u,\rho)\in H$.
\end{pr}

\begin{proof} (i) Let $b\in B$. It suffices to show that $bu\in Ru$ for each $u\in U$,
for then the proposition follows by applying this to the modules $U^n$ ($n\in\bn$) instead of $U$. By the hypothesis 
$u\in Tu$ and so $Ru=Tu$. By Lemma \ref{le11} $T$ is a left
ideal in $B$, hence $Tu$ is an $B$-submodule of $U$. So  $bRu=bTu\subseteq Ru$, 
hence in particular $bu\in Ru$.

(ii) Let $(u_i,\rho_i)$ ($i=1,\ldots,n$) be elements of $H$ and denote
$$z:=(\inner{u_1}{\rho_1s},\ldots,\inner{u_n}{\rho_ns})\in Z^n,\ \ 
V:=\{(\inner{u_1}{\rho_1r},\ldots,\inner{u_n}{\rho_nr}),\, r\in R\}\subseteq Z^n.$$
We have to prove that $z\in V$ and for this we will show that the assumption $z\notin V$ leads to a contradiction. If $z\notin V$, then (since
$Z$ is an injective cogenerator in $\rmod{C}$ and $V$ is a $C$-submodule of $Z^n$), there exists $c\in\lhom{C}{Z^n}{Z}$ such that $cV=0$ and
$cz\ne0$. Since $\lend{C}{Z}=C$ by the hypothesis, $c=(c_1,\ldots,c_n)$ for some elements $c_j\in C$ and we have now
$$\sum_j\inner{u_jc_j}{\rho_jr}=0\ \mbox{for all}\ r\in R,\ \mbox{while}\ \sum_j\inner{u_jc_j}{\rho_js}=cz\ne0.$$
Denoting $u:=(u_1c_1,\ldots,u_nc_n)\in U^n$ and $\rho:=(\rho_1,\ldots,\rho_n)\in (U^*)^n=(U^n)^*$, this means that
\begin{equation}\label{201}\inner{u}{\rho R}=0\ \ \mbox{and}\ \ \inner{u}{\rho s}\ne0.\end{equation}

To prove that (\ref{201}) leads to a contradiction, we use that $v\in Tv$ for all 
$v\in U^n$. This (together with the definition
of $T$) implies that every (cyclic, hence every) $R$-submodule $U_0$ of $U^n$ is equal to the sum of the images of all $R$-module homomorphisms 
from $U^n$ to $U_0$. Applying this to $U_0=(\rho R)_{\perp}$, the annihilator of $\rho R$ in $U^n$, it follows that for each 
$\omega\in (U^n)^*\setminus ((\rho R)_{\perp})^{\perp}$
there exists a nonzero
$R$-module homomorphism $f:U^n\to (\rho R)_{\perp}/(\rho R+\omega R)_{\perp}$. Then extending the adjoint $f^*$ (using the $R$-injectivity of
$(U^*)^n$, a consequence of the flatness of $U$) we find an $R$-endomorphism $g:(U^n)^*\to (U^n)^*$ such that $g(\rho R)=0$ and $g(\omega)\ne0$.
This means that $((\rho R)_{\perp})^{\perp}$ is the intersection of kernels of all $g\in\lend{R}{(U^n)^*}$ satisfying $g(\rho R)=0$. Since such kernels
are invariant under each biendomorphism $s\in\tilde{B}$ (because $s$ and $g$ commute), it follows that $((\rho R)_{\perp})^{\perp}$ is invariant under $s$;
in particular $\rho s\in((\rho R)_{\perp})^{\perp}$. But this contradicts (\ref{201}).
\end{proof}

\section{$\biend{R}{R^*}$ and Arens products on $R^{**}$}

Throughout this section we assume that $R$ is a $C$-algebra.
The second dual $R^{**}$ acts on $R^*$  from the right and from the left side as 
\begin{equation}\label{B2}\inner{r}{\rho s}:=\inner{s}{r \rho}\ \ \ \mbox{and}\ \ \ \inner{r}{s\rho}:=\inner{s}{\rho r}
\ \ \ (r\in R,\ \rho\in R^*, s\in R^{**}),\end{equation}
where (we recall) $r\rho$ and $\rho r$ are defined by $\inner{x}{r\rho}=\inner{xr}{\rho}$ and $\inner{x}{\rho r}=
\inner{rx}{\rho}$ ($x\in R$).
As it is well-known from  Banach algebra theory \cite[1.4.1]{P},  there are two associative algebra products on $R^{**}$
that extend the product on $R$ and were defined by Arens as follows:
\begin{equation}\label{B1}\inner{s\cdot t}{\rho}:=\inner{s}{t\rho}\ \ \ \mbox{and}\ \ \ \inner{s\diamond t}{\rho}:=
\inner{t}{\rho s}\ \ \ (s,t\in R^{**},\ \rho\in R^*).\end{equation}
$R$ is called {\em Arens $Z$-regular}  (where $Z$ is the injective cogenerator in $\rmod{C}$ relative to
which the duality is defined) if the two products $\cdot$ and $\diamond$ coincide on $R^{**}$. When $Z$ is the minimal 
injective cogenerator in $\rmod{C}$ (that is, the injective hull of the direct sum of `all' simple $C$-modules 
(see \cite[19.13]{La})
then we simply say that $R$ is {\em Arens regular}.

It can be verified that the left multiplication $\lambda_t$ on $R^*$ by each element
$t\in R^{**}$ commutes with the right multiplication by every $r\in R$, hence the set $\lambda_{R^{**}}$ of all such 
multiplications is contained in
the endomorphism ring $\lend{R}{R^*}$ of $R^*\in\rmod{R}$. Conversely, each $\phi\in\lend{R}{R^*}$ commutes with
the right multiplication on $R^*$ by each $r\in R$, therefore $\phi^*$ commutes with
the left multiplication by $r$ on $R^{**}$, hence $\phi^*|R$ must be the right multiplication by the 
element $t:=\phi^*(1)$. Then $\inner{(\phi-\lambda_t)(R^*)}{R}=\inner{R^*}{(\phi^*-\lambda_t^*)(R)}=0$, hence $\phi=\lambda_t$.
This shows that $\lend{R}{R^*}=\lambda_{R^{**}}$. 

Similarly, since a biendomorphism $\psi\in\biend{R}{R^*}$ must commute 
with all elements of $\lambda_{R^{**}}=\lend{R}{R^*}$, $\psi^*$ commutes with $\lambda_u^*$ for each $u\in R^{**}$.
But directly from the definitions we can see that $\lambda_u^*$ is given by $\inner{t}{\lambda_u^*}=t\cdot u$ for all
$t\in R^{**}$, hence $\psi^*(t\cdot u)=\psi^*(t)\cdot u$ and it follows that $\psi^*(u)=s\cdot u$, where
$s=\psi^*(1)$. Now for each $\rho\in R^*$ and $r\in R$ we have $\inner{r}{\rho s}=\inner{s}{r\rho}=
\inner{s\cdot r}{\rho}=\inner{\psi^*(r)}{\rho}=\inner{r}{\psi(\rho)}$, hence $\psi$ must be the right multiplication
on $R^*$ by 
the element $s\in R^{**}$. Moreover, since $\psi$ commutes with $\lambda_{R^{**}}$, $s$ must satisfy 
$(t\rho)s=t(\rho s)$ for all $t\in R^{**}$ and $\rho\in R^*$,
that is
\begin{equation}\label{B3}\inner{s}{rt\rho}=\inner{t}{\rho sr}\ \ \mbox{for all}\ r\in R.\end{equation}
With $r=1$ (\ref{B3}) and (\ref{B1}) show that $\inner{s\cdot t}{\rho}=\inner{s}{t\rho}=\inner{t}{\rho s}=
\inner{s\diamond t}{\rho}$
for all $\rho\in R^*$, hence 
\begin{equation}\label{B4}s\cdot t=s\diamond t\ \ \mbox{for all}\ t\in R^{**}.\end{equation}
Conversely, (\ref{B4}) means that $\inner{s}{t\rho}=\inner{t}{\rho s}$ for all $\rho\in R^*$; replacing in this equality
$t$ by $rt$ ($r\in R$) we obtain (\ref{B3}). Elements $s\in R^{**}$ satisfying (\ref{B4}) constitute a subring of 
$(R^{**},\cdot)$
(and of $(R^{**},\diamond)$)
which in Banach algebra theory is called the {\em left topological center} of $R^{**}$ \cite[2.24]{DLS}. 
This proves the first part of the following proposition. 

\begin{pr}\label{prb} Regard $R^*$ as a right $R$-module in the usual way.

(i)  The ring $\lend{R}{R^*}$ consists of left multiplications by all elements of $R^{**}$, 
while  $\biend{R}{R^*}$ consists of right multiplications by all elements $s\in R^{**}$
satisfying $s\cdot t=s\diamond t$ for all $t\in R^{**}$. 

(ii) $\biend{R}{R^*}$ consists of right multiplications by elements $s\in R^{**}$ such that for each finite subset
$F$ of $R^*$ there exists an element $r_F\in R$ satisfying $\rho s=\rho r_F$ for all $\rho\in F$.

(iii) If $R$ is linearly compact in $\rmod{R}$, then $\biend{R}{R^*}=R$.
\end{pr}

\begin{proof} We have already proved (i). Then part (ii) follows from the fact that $R^*$ is a cogenerator in $\rmod{R}$ since for
any cogenerator $M\in\rmod{R}$ the ring $R$ is dense in $\biend{R}{M}$ (that is, on finite subsets of $M$ each
biendomorphism coincides with the multiplication by an element of $R$) by  \cite[p. 164]{AF}. (To see that
$R^*$ is indeed a cogenerator, given $X\in\rmod{R}$, let $G$ be so large that there is an epimorphism $q:R^{(G)}\to X^*$ and let
$X\to X^{**}$ be the natural map. Then
we have a monomorphism $X\to X^{**}\stackrel{q^*}{\rightarrow}(R^*)^G$.)

(iii) Let $s\in\biend{R}{R^*}$. By (ii) for each finite subset $F$ of $R^*$ there exists an element
$r_F\in R$ such that $\rho r_F=\rho s$. Let $J_F=\ann{R}{F}$, a right ideal in $R$. For each finite subset
$F=\{\rho_1,\ldots,\rho_n\}$ of $R^*$ we have that $r_F-r_{\{\rho_j\}}\in J_{\{\rho_j\}}$. Then by the definition of 
linear compact modules 
there exists an $r\in R$ such that $r-r_{\{\rho\}}\in J_{\{\rho\}}$ for all $\rho\in R^*$. This implies that
$\rho s=\rho r_{\{\rho\}}=\rho r$, hence $s=r$.
\end{proof}

\begin{pr}If $\biend{R}{R^*}=R$ (where $R^*$ is regarded as a right $R$-module) then $\biend{R}{U^*}=R$ for every 
generator $U\in\lm{R}$.
\end{pr}

\begin{proof}Since $U$ is a generator there exists an epimorphism $U^n\to R$ (for some $n\in\mathbb{N}$), hence $U^n=R\oplus V$
for a submodule $V$ of $U$. Then $(U^*)^n=R^*\oplus V^*$. Since a biendomorphism $\phi\in\biend{R}{(U^*)^n}$ commutes
with the projections of $(U^*)^n$ onto the two summands $R^*$ and $V^*$, $\phi$ must be of the form $\phi=a\oplus\psi$,
where $a\in\biend{R}{R^*}$ and $\psi\in\biend{R}{V^*}$. For each $v\in V$ let $f_v\in\lhom{R}{R}{V}$ be defined by
$\inner{r}{f_v}:=rv$ and let $g\in\lend{R}{R^*\oplus V^*}$ be represented by the matrix
$$g=\left[\begin{array}{ll}
0&f_v^*\\
0&0\end{array}\right].$$
Then $\phi\circ g=g\circ\phi$, meaning that  $af_v^*=f_v^*\psi$, that is $\inner{\rho\circ f_v}{a}=\inner{\rho}{\psi}\circ f_v$ for all $\rho\in V^*$. 
By assumption $a$ is the right multiplication by an element
$r_0\in R$, hence the last equality means that $(\rho\circ f_v)r_0=\inner{\rho}{\psi}\circ f_v$ (the equality of two elements of
$R^*$), that is $\inner{r_0rv}{\rho}=\inner{rv}{\inner{\rho}{\psi}}=\inner{v}{\inner{\rho }{\psi}r}$ for all $r\in R$. Evaluating at $r=1$ 
we conclude that
$\inner{v}{\inner{\rho}{\psi}}=\inner{r_0v}{\rho}=\inner{v}{\rho r_0}$ for all $v\in V$ and $\rho\in V^*$, hence $\psi$
must be the right multiplication by $r_0$ on $V^*$. Consequently $\phi$ is the right multiplication by $r_0$ on 
$(U^*)^n$, which proves that $\biend{R}{(U^*)^n}=R$. Since $\biend{R}{U^*}=\biend{R}{(U^*)^n}$, this concludes the proof.
\end{proof}

Arens regular Banach algebras are characterized by the weak compactness of certain operators \cite{P} and characterizations in the same spirit
are known for certain topological algebras. Here we would like to characterize Arens regularity of $C$-algebras 
(relative to the given injective cogenerator $Z$) in purely algebraic terms. For this we first need some facts concerning extensions of 
bilinear forms. 

Given $C$-modules $X$ and $Y$, a $C$-bilinear map $\theta:X\times Y\to Z$ defines two $C$-module homomorphisms
$$t_l:X\to Y^*,\ t_l(x)(y):=\theta(x,y)\ \ \mbox{and}\ \ t_r:Y\to X^*,\ t_r(y)(x):=\theta(x,y).$$
Using the second adjoint $t_l^{**}:X^{**}\to Y^{***}$ we can define the left extension $\theta_l:X^{**}\times Y^{**}\to Z$ by
$$\theta_l(x^{**},y^{**}):=(t_l^{**}(x^{**}))(y^{**})\ \ (x^{**}\in X^{**},\ y^{**}\in Y^{**}).$$
Here $Y^{***}:=\lhom{\bcom{C}}{Y^{**}}{Z}$. Let us say that a $C$-bilinear map $\psi:X^{**}\times Y^{**}\to Z$ is {\em normal in the first variable} if 
every  map $\psi_{y^{**}}:X^{**}\to Z$, $\psi_{y^{**}}(x^{**}):=\psi(x^{**},y^{**})$ ($y^{**}\in Y^{**}$) is an 
evaluation at 
an element $x^*(y^{**})$ of $X^{*}$, that is, $\psi(x^{**},y^{**})=\inner{x^{**}}{x^*(y^{**})}$ for all $x^{**}\in X^{**}$.
Normality in the second variable is defined in the same way and $\psi$ is called {\em normal} if it is normal in each variable separately. 
The extension $\theta_l$ of $\theta$ is normal in the first variable since $\theta_l(x^{**},y^{**})=\inner{t_l^{**}(x^{**})}{y^{**}}=
\inner{x^{**}}{t_l^*(y^{**})}$
and $t_l^*(y^{**})\in X^*$. It can be easily seen that $\theta_l|(X^{**}\times Y)$ is the only $C$-bilinear extension of $\theta$ to $X^{**}\times Y$ which is normal
in the first variable. Similarly we can define the right extension $\theta_r$ of $\theta$ to $X^{**}\times Y^{**}$ by
$\theta_r(x^{**},y^{**}):=(t_r^{**}(y^{**}))(x^{**})$, and $\theta_r$ is normal in the second variable for each fixed 
$x^{**}\in X^{**}$. If the equality $\theta_l=\theta_r$ holds, then $\theta_l$ is
the (necessarily unique) normal extension of $\theta$ to $X^{**}\times Y^{**}$. Conversely, assume that there exists a normal $C$-bilinear extension
$\hat{\theta}:X^{**}\times Y^{**}\to Z$ of $\theta$. Then, since also $\theta_l(x,y^{**})$ is normal in the second variable for each $x\in X$
(namely, $\theta_l(x,y^{**})=\inner{t_l^{**}(x)}{y^{**}}=\inner{t_l(x)}{y^{**}}$ and $t_l(x)\in Y^*$) the equality 
$\hat{\theta}(x,y^{**})=\theta_l(x,y^{**})$
must hold for all $x\in X$ and $y^{**}\in Y^{**}$. But then, by normality in the first variable, $\hat{\theta}=\theta_l$. By symmetry we thus see that
$\theta$ has a normal extension to $X^{**}\times Y^{**}$ if and only if $\theta_l=\theta_r$. To better explain the meaning of this equality, let
$p_{Y^*}:Y^{***}\to Y^{*}$ be the adjoint of the
natural map $\iota_Y:Y\to Y^{**}$ and define an extension of $\theta$ by 
$$\tilde{\theta}(x^{**},y^{**}):=\inner{p_{Y^*}(t_l^{**}(x^{**}))}{y^{**}}=\inner{x^{**}}{(p_{Y^*}t_l^{**})^*(y^{**})}=
\inner{x^{**}}{(t_l^*\iota_Y)^{**}(y^{**})}.$$
Since $t_l^*\iota_Y=t_r$ (by a straightforward verification), it follows that $\tilde{\theta}=\theta_r$. Thus the equality $\theta_l=\theta_r$
means that $\inner{t_l^{**}(x^{**})}{y^{**}}=\theta_l(x^{**},y^{**})=\tilde{\theta}(x^{**},y^{**})=\inner{p_{Y^*}(t_l^{**}(x^{**}))}{y^{**}}$,
that is $t_l^{**}(x^{**})=p_{Y^*}(t_l^{**}(x^{**}))\in Y^*$.
We state these conclusions as a lemma.

\begin{lemma}\label{le201}A $C$-bilinear map $\theta:X\times Y\to Z$ can be extended to a normal $C$-bilinear map 
$\hat{\theta}:X^{**}\times Y^{**}\to Z$ if and only if the extension $\theta_l$ is normal in the second variable and this is equivalent
to the condition $\theta_l=\theta_r$ and also to  $t_l^{**}(X^{**})\subseteq Y^*$ (and, symmetrically,
to $t_r^{**}(Y^*)\subseteq X^*$).
\end{lemma}

\begin{lemma}\label{le202}Suppose that $Z$ is injective as a left module over $\com{C}=\lend{C}{Z}$ (in addition to being injective and
a cogenerator in $\rmod{C}$). Then a homomorphism $f:X\to Y$ of right $C$ modules satisfies 
$f^{**}(X^{**})\subseteq Y$ if and only if
$f(X)$ is a reflexive module (that is, $f(X)^{**}=f(X)$) and this is the case if and only if $f(X)$ is linearly compact in $\rmod{C}$.
\end{lemma}

\begin{proof} By Lemma \ref{le301}(ii) we have 
$f(X)^{**}=f(X)^{\perp\perp}=f^{**}(X^{**})$. Therefore, if $f(X)$ is reflexive, $f^{**}(X^{**})=f(X)^{**}=f(X)\subseteq Y$.
For the converse, first note that for any $C$-submodule $V$ of $Y$, denoting by $\kappa:V\to Y$ the inclusion, the map 
$\kappa^{**}$
is injective (as can be seen by using Lemma \ref{le301}), so $V^{**}$ can be regarded as a submodule of $Y^{**}$. 
Moreover $V^{**}\cap Y=V$. (Indeed, suppose that $y\in (V^{**}\cap Y)\setminus V$.
Then there exists $y^*\in Y^*$ such that $y^*(V)=0$ and $y^*(y)\ne0$. But this implies that $y\notin V^{\perp\perp}$, which is a contradiction
since $V^{\perp\perp}=V^{**}$ by Remark \ref{re0}.) Applying this to $V=f(X)$ and using the equality $f(X)^{**}=f^{**}(X^{**})$, the assumption 
$f^{**}(X^{**})\subseteq Y$ implies that $f(X)^{**}=f(X)^{**}\cap Y=f(X)$, proving that $f(X)$ is reflexive. The reflexivity of $f(X)$ is equivalent
to linear compactness by M\" uler's first theorem \cite[19.66]{La}.
\end{proof}

\begin{theorem}\label{th203}Let $X,Y$ be $C$-modules, $Z\in\rmod{C}$ an injective cogenerator in terms of which the duals
are defined, $\theta:X\times Y\to Z$ a $C$-bilinear map and $t_l$, $t_r$ the associated operators defined above. Assume that
$Z$ is injective also in $\lm{\com{C}}$ (recall that $\com{C}=\lend{C}{Z}$).
Then $\theta$ can be extended to a normal $C$-bilinear map $\hat{\theta}:X^{**}\times Y^{**}\to Z$  if and
only if $t_l(X)$ is a reflexive (or, equivalently, linearly compact) $C$-module. (By symmetry this is the case
if and only if $t_r(Y)$ is a reflexive $C$-module.)
\end{theorem}

\begin{proof}By Lemma \ref{le201} $\theta$ can be extended to a normal $C$-bilinear map on $X^{**}\times Y^{**}$ if and only if
$t_l^{**}(X^{**})\subseteq Y^*$. By Lemma \ref{le202} (applied to the $C$-module $Y^*$ instead of $Y$) this is equivalent to 
the reflexivity of $t_l(X)$.
\end{proof}

\begin{re}Consider the topology on $X^{**}$ in which the basic neighborhoods of each $x_0^{**}\in X^{**}$ are all sets of the form
$$\mathcal{N}(x_0^{**};x_1^*,\ldots,x_n^*):=\{x^{**}\in X^{**}:\, \inner{x_j^*}{x^{**}-x_0^{**}}=0\ \ (j=1,\ldots,n)\},$$
where $\{x_1^*,\ldots,x_n^*\}$ runs through all finite subset of $X^*$. 
If $Z$ is faithfully balanced as a $\com{C},C$-bimodule (that is, if $\bcom{C}:=\lend{\com{C}}{Z}=C$), then $X^*$ contains precisely those
elements $x^{***}\in X^{***}$ which are continuous as maps from $X^{**}$ to $Z$, where $Z$ is equipped with the discrete topology.
({\em Proof.} Such $x^{***}$ must map a neighborhood $\mathcal{N}(0;x_1^*,\ldots,x_n^*)$ of $0$ into $0$. This means that $\ker x^{***}$ contains
the intersection $\cap_{j=1}^n\ker x_j^*$, hence $c_0:(\inner{x_1^*}{x^{**}},\ldots,\inner{x_n^*}{x^{**}})\mapsto\inner{x^{***}}{x^{**}}$ is
a well defined homomorphism of right $\bcom{C}$-modules with the domain 
$\{(\inner{x_1^*}{x^{**}},\ldots,\inner{x_n^*}{x^{**}}):\, 
x^{**}\in X^{**}\}$ and range contained in $Z$. Since $\bcom{C}=C$, by the $C$-injectivity of $Z$, $c_0$ can be extended to
a homomorphism $c:Z^n\to Z$, given by a $n$-tuple $(c_1,\ldots,c_n)$ of elements $c_j\in C$. Thus
$$\inner{x^{***}}{x^{**}}=(\inner{x_1^*}{x^{**}},\ldots,\inner{x_n^*}{x^{**}})c=\sum_j \inner{x_j^*}{x^{**}}c_j=\inner{\sum_jc_jx_j^*}{x^{**}}$$
for all $x^{**}\in X^{**}$, proving that $x^{***}=\sum_jc_jx_j^*\in X^*$.) In this case normality of a $C$-bilinear map 
$\psi:X^{**}\times Y^{**}\to Z$ means continuity in each variable separately. Since $X$ and $Y$ are dense in $X^{**}$ and $Y^{**}$ by Lemma \ref{le301}(iii)
$\theta_l$ and $\theta_r$ are the unique extensions of a $C$-bilinear map $\theta:X\times Y\to Z$ that are continuous in the first and in the second variable
(respectively). Moreover, the condition $\theta_l(x^{**},y^{**})=\theta_r(x^{**},y^{**})$ can be expressed as follows: 
choose nets $(x_F)$ and $(x_G)$ in
$X$ and $Y$ converging to $x^{**}$ and $y^{**}$ (respectively); then
$\lim_{F}\lim_{G}\theta(x_F,y_G)=\lim_{G}\lim_{G}\theta(x_F,x_G)$. This is analogous to the well-known situation in Banach algebra theory. 
\end{re}

Returning now to $C$-algebras $R$, by Proposition \ref{prb}(i) $R$ is  Arens $Z$-regular if and only if $\biend{R}{R^*}=R^{**}$, that is
$s(\rho t)=(s\rho)t$ for all $s,t\in R^{**}$ and $\rho\in R^*$. This means that 
\begin{equation}\label{301}\inner{r}{s(\rho t)}=\inner{r}{(s\rho) t}\ \ \mbox{for all}\ r\in R.\end{equation}
For fixed $r\in R$ and $\rho\in R^*$ both sides of (\ref{301}) are $C$-bilinear maps $R^{**}\times R^{**}\to Z$ which extend the $C$-bilinear map
$\theta:R\times R\to Z$ defined by $\theta(a,b)=\inner{r}{a\rho b}$. The left side of (\ref{301}) is normal in the variable $s$ since
$\inner{r}{s(\rho t)}=\inner{s}{\rho t r}$, where $\rho tr\in R^*$. Similarly the right side of (\ref{301}) is normal in the variable $t$, so
(\ref{301}) means that $\theta_l=\theta_r$. By Lema \ref{le201} and Theorem \ref{th203} this means that the range of the map $t_l:R\to R^*$, $\inner{b}{t_l(a)}=\theta(a,b)$
($a,b\in R$) is reflexive. From $\theta(a,b)=\inner{r}{a\rho b}=\inner{b}{ra\rho}$ we now observe that $t_l(a)=ra\rho$. Thus the 
Arens $Z$-regularity of $R$ is equivalent to the condition that all $C$-modules $rR\rho$ ($r\in R$, $\rho\in R^*$) are 
reflexive. If we assume that $Z$ is injective also as a left $\com{C}$-module, then reflexive $C$-modules are the same as
linearly compact $C$-modules by \cite[19.66]{La}, \cite{CFu}. Since 
$rR\rho$ is a $C$-submodule of $R\rho$ and submodules of linearly compact modules are linearly compact, this proves 
most of the following theorem.

\begin{theorem}\label{th204}Suppose that $Z$ is injective in $\lm{\com{C}}$ (in addition to being an injective cogenerator
in $\rmod{C}$). Then $R$ is Arens $Z$-regular if and only if $R\rho$ is reflexive (that is, linearly compact) $C$-module for each $\rho\in R^*$.
(By symmetry, this is also equivalent to the reflexivity of $\rho R$ for all $\rho$.) In particular if $C$ is linearly
compact in $\rmod{C}$, $R$ is Arens regular if and only if $R\rho$ is linearly compact in $\rmod{C}$. 
\end{theorem}

\begin{proof}It only remains to prove the last sentence of the theorem. Since $C$ is commutative and linearly compact,
$C$ admits a Morita duality by a result of Anh \cite{A}, \cite[19.77]{La}. Then by M\" uler's second theorem 
\cite[19.71]{La} the minimal injective cogenerator $Z\in\rmod{C}$ defines a Morita duality from $C$ to $\com{C}$, hence
by Morita's first theorem \cite[19.43]{La} $Z$ is injective (and a faithfully balanced cogenerator) in $\lm{\com{C}}$,
so the first part of the theorem applies.
\end{proof}

\begin{co}\label{co205}If $C$ is linearly compact in $\rmod{C}$, then $C$-subalgebras of an Arens regular $C$-algebra $R$ are Arens regular.
\end{co}

\begin{proof}Let $S$ be a $C$-subalgebra of $R$ and $\omega\in S^*$. By the $C$-injectivity of $Z$ we can extend
$\omega$ to $\rho\in R^*$. Then the map $f:R\rho\to S^*$, $f(r\rho):=(r\rho)|S$, is homomorphism of $C$-modules
and $S\omega\subseteq f(R\rho)$. Since quotients of reflexive modules are reflexive by \cite[19.58]{La} (we are here
again in the context of Morita duality as in the proof of Theorem \ref{th204}), $f(R\rho)$ is reflexive (as a $C$-module)
and then its $C$-submodule $S\omega$ is reflexive.
\end{proof}

\section{Arens regularity for algebras over fields}

Let us now consider the special case when $C$ is a field $\f$ and $Z=\f$. Then Theorem \ref{th204} says that $R$ is Arens regular if and only if
$R\rho$ (and $\rho R$) are finite dimensional vector spaces for all $\rho\in R^*$. If, for a given $\rho$, $\{\rho_1,\ldots,\rho_n\}$ is a basis of
$R\rho$, then $N:=\cap_{j=1}^n\ker\rho_j$ is a subspace of finite codimension in $R$ such that $(r\rho)(N)=0$ for all $r\in R$. This means that
$\ker\rho$ contains the right ideal $NR$. Conversely, if $\ker\rho$ contains a right ideal of finite codimension, then clearly $R\rho$ is finite dimensional.
Further, finite dimensionality of $R\rho$ and $\rho R$ for all $\rho\in R^*$ is easily seen to be equivalent to finite 
dimensionality of $R\rho R$
and this in turn means that $\ker\rho$ contains a two-sided ideal of finite codimension in $R$. 
The intersection $J$ of all
such ideals must be $0$. (Otherwise we could find $\rho\in R^*$ with $\rho(J)\ne0$ and then an ideal $J_{\rho}\subseteq
\ker\rho$ of finite codimension, hence $J\subseteq J_{\rho}$, a contradiction with $\rho(J)\ne0=\rho(J_{\rho})$.) This proves
the following proposition.

\begin{pr}\label{prA}For an algebra $R$ over a field $\f$ the following are equivalent:

(i) $R$ is Arens regular.

(ii) $\dim R\rho<\infty$ (or, symmetrically, $\dim\rho R<\infty$) for each $\rho\in R^*$.

(iii) The kernel of each $\rho\in R^*$ contains a left ideal of finite codimension.

(iv) The kernel of each $\rho\in R^*$ contains a two-sided ideal of finite codimension.

\end{pr}

Proposition \ref{prA} restrict the class of Arens regular algebras much more than it might appear on the first sight.

\begin{lemma}\label{leA}If $R$ is Arens regular algebra over a field $\f$ then there exists $n_0\in\bn$ such that $\dim{R\rho}\leq n_0$ for
all $\rho\in R^*$.
\end{lemma}

\begin{proof}Suppose the contrary, that for each $n\in\bn$ there exists a $\rho_n\in R^*$ with $\dim{R\rho_n}\geq n$.
Then for each $n$ we can choose $y_{n,j}\in R$ such that the set $\{y_{n,j}\rho_n:\, j=1,\ldots,n\}$ is linearly 
independent,
which means that the vectors $$(\rho_n(xy_{n,j}))_{x\in R}\in \f^R\ \ (j=1,\ldots,n)$$ are linearly independent. Then, by a standard argument
there exist $x_{n,j}$ ($j=1,\ldots,n$) such that the determinant of the $n\times n$ matrix $[\rho_n(x_{n,i}y_{n,j})]$ is nonzero,
so the restrictions of the functionals $y_{n,j}\rho_n$ ($j=1,\ldots,n$) to the subalgebra $R_n$ of $R$ generated by $\{x_{n,i},y_{n,j}:\, i,j=1,\ldots,n\}$
are linearly independent. Considering the algebra $S$ generated by the union of all $R_n$, we see that $S$ is of countable dimension
over $\f$ and the dimensions $\dim S\rho$ ($\rho\in S^*$) are not bounded. So, replacing $R$ by $S$, we may assume that $R$ has countable dimension
over $\f$. Let $(a_k)_{k\in\bn}$ be a basis of $R$ as a vector space and
$$a_ia_j=\sum_k\mu_{i,j,k}a_k\ \ \ (\mu_{i,j,k}\in\f),$$
where for fixed $i,j$ only finitely many coefficients $\mu_{i,j,k}$ are nonzero. Identify $R^*$ with $\f^{\bn}$ by identifying each 
$\rho\in R^*$ with the sequence $(\rho(a_k)_k)\in\f^{\bn}$. By Proposition \ref{prA} $\dim R\rho<\infty$ and this is equivalent to the condition that
the linear span of $\{((a_j\rho)(a_i))_{i\in\bn}\in\f^{\bn}:\ j\in\bn\}$ is finite dimensional, hence to the condition that the infinite
matrix $[\rho(a_ia_j)]_{i,j\in\bn}$ has finite rank. Since this matrix is equal to
$$[\rho(a_ia_j)]_{i,j\in\bn}=\sum_k[\mu_{i,j,k}]_{i,j}\rho(a_k),$$
we see that the map
\begin{equation}\label{500}\mu:\f^{\bn}\to\mathbb{M}_{\bn}(\f),\ \ \mu((\rho_k)_{k\in\bn}):=\sum_k[\mu_{i,j,k}]_{i,j}\rho_k\end{equation}
contains in its range only finite rank matrices. Let $\f$ be equipped with the discrete topology, which is metrizable by $d(\alpha,\beta)=1$ if 
$\alpha\ne\beta$, so that $\f^{\bn}$ with the product topology is also metrizable by $d((\alpha_k),(\beta_k))=\sum_k2^{-k}d(\alpha_k,\beta_k)$.
Then for each $n$ the set $F_n\subseteq\f^{\bn}$, consisting of all $\rho\in\f^{\bn}$ such that rank of $\mu(\rho)$ is at most $n$, is closed.
(To see this, note that if a matrix $\mu(\rho)$ has rank $\geq n+1$, then some finite submatrix of it, say $I\times J$
submatrix, has rank $\geq n+1$; but since only finitely
many indexes $i,j$ are involved in such a submatrix, it follows that the sum 
$\sum_k[\mu_{i,j,k}]_{(i,j)\in I\times J}\rho_k$ involves only finitely many
components $\rho_k$ of $\rho$, say components with $k\leq k_0$. Then the definition (\ref{500}) of $\mu$ implies that the rank of $\mu(\omega)\geq n+1$ if $\omega$ agrees 
with $\rho$ in the first $k_0$ components.) Since $\cup_nF_n=\f^{\bn}$ by Proposition \ref{prA}, it follows by Baire's theorem that 
$F_n$ has nonempty interior for some $n\in\bn$. If $\rho=(\rho_k)_k$ is an interior point of $F_n$, then there exists $m\in\bn$ such that
every $\omega=(\omega_k)\in\f^{\bn}$ satisfying $\omega_k=\rho_k$ for $k\leq m$ must be in $F_n$. Since a general $\tau\in\f^{\bn}$ can be decomposed
as $\tau=\sum_{k\leq m}\sigma_k\lambda_k+\omega$, where $\lambda_k:\f^{\bn}\to\f$ are the coordinate functionals (that is, $\lambda_k(a_i)=\delta_{k,i}$),
$\sigma_k\in\f$ and $\omega_k=\rho_k$ for $k\leq m$, it follows that
$$\mu(\tau)=\sum_{k\leq m}\sigma_k[\mu_{i,j,k}]_{i,j}+\mu(\omega)$$
has rank at most $\sum_{k\leq m}{\rm rank}([\mu_{i,j,k}]_{i,j})+n$. This proves that 
$$n_0:=\sup_{\tau\in\f^{\bn}}{\rm rank}(\mu(\tau))<\infty.$$ It follows that for each $\rho\in R^*$ the rank of the matrix 
$[\rho(a_ia_j)]_{i,j\in\bn}$ is at most $n_0$, which implies that at most $n_0$ among the vectors $((a_j\rho)(a_i))_{i\in\bn}\in\f^{\bn}$
are linearly independent, that is, $\dim R\rho\leq n_0$.
\end{proof}

\begin{theorem}\label{thA}An infinite dimensional algebra $R$ over a field $\f$ is Arens regular if and only if $R$ 
contains an ideal $J$ of finite codimension such that $J^2=0$.
\end{theorem}

\begin{proof}Suppose that $R$ contains an ideal $J$ of finite codimension with $J^2=0$.
Let $X$ be a vector subspace of $R$ such that $R=X\oplus J$. Given $0\ne\rho\in R^*$, let $q:R\to R/\ker\rho=\f$
be the quotient map. For each $x\in X$ the subspace $J_x:=\{a\in J:\, xa\in\ker\rho\}$ is of finite codimension
since $J_x$ is just the kernel of the map $q\lambda_x$, where $\lambda_x$ is the left multiplication by $x$ on $J$.
Since $X$ is finite dimensional, it follows that there exists a subspace $J_{\rho}$ of finite codimension in $J$ such
that $XJ_{\rho}\subseteq\ker\rho$. (Namely, $J_{\rho}=\cap_{j=1}^nJ_{x_j}$, where $\{x_1,\ldots,x_n\}$ is a basis
of $X$.) Since $R=X\oplus J$ and $J^2=0$, it follows that $RJ_{\rho}\subseteq\ker\rho$. 
Thus $RJ_{\rho}$ is a left ideal of finite codimension contained in $\ker{\rho}$, hence $R$ is Arens regular by
Proposition \ref{prA}.

To prove the converse, it is convenient
to choose three basis $(a_i)_{i\in I}$, $(b_i)_{i\in I}$ and $(c_i)_{i\in I}$ for $R$ as a vector space over $\f$, where 
the index set $I=\{1,2,\ldots\}$ is well
ordered, but not necessarily countable. 
(In the beginning these may be the same
basis, but during the course of the proof they will change.)  Let
\begin{equation}\label{501}b_ic_j=\sum_k\mu_{i,j,k}a_k\ \ (\mu_{i,j,k}\in\f)\end{equation}
be the multiplication table of $R$. Identify $R^*$ with $\f^I$ by $\rho\mapsto(\rho(a_i))_{i\in I}$ and 
consider the map $\mu:\f^I=R^*\to\mathbb{M}_I(\f)$ defined by 
$$\mu((\rho)_{k})_{k\in I}=\sum_{k\in I}[\mu_{i,j,k}]_{i,j\in I}\rho_k,$$
where (for fixed $i,j\in I$) only finitely many $\mu_{i,j,k}$ are nonzero. Let $\omega=(\omega_k)\in\f^{I}$ be such
that ${\rm rank}(\mu(\omega))=n_0=\max_{\rho\in\f^{I}}{\rm rank}(\mu(\rho))$ (Lemma \ref{leA}). By a suitable replacement of the basis $(a_i)$ 
we may assume that $\omega$
is just the first coordinate functional $\lambda_1$, where $\lambda_i\in R^*$ are defined by $x=\sum_{i\in I}\lambda_i(x)a_i$.
(For example, if $\omega_1\ne0$, we may replace $a_i$ by $(a_i^{\prime})$,
where $a_1=\omega_1a_1^{\prime}$ and $a_i=a_i^{\prime}+\omega_ia_1^{\prime}$ for $i>1$.)
Thus ${\rm rank}(\mu(\lambda_1))=n_0$.
If we change the basis $(b_i)_{i\in I}$ and $(c_i)_{i\in I}$, it is easy to show
that each of the matrices $[\mu_{i,j,k}]_{i,j\in I}$ changes to an equivalent matrix and we have the elementary row and column
operations at our disposal, so we may assume that the matrix $[\mu_{i,j,1}]_{i,j\in I}=\mu(\lambda_0)$
has the identity matrix $1$ in its $n_0\times n_0$ upper left corner and zeros elsewhere. 

Now fix a $k>1$ and consider the decomposition of the matrix
$$[\mu_{i,j,k}]_{i,j}=\left[\begin{array}{cc}
a&b\\
c&d\end{array}\right],$$
where $a$ is a $n_0\times n_0$ matrix, $d$ is a $(I\setminus \{1,\ldots,n_0\})\times (I\setminus \{1,\ldots,n_0\})$ matrix and so on. From the definition
of $n_0$ we have that ${\rm rank}(t[\mu_{i,j,1}]_{i,j}+[\mu_{i,j,k}]_{i,j})\leq n_0$ for all $t\in\f$, that is
\begin{equation}\label{502} {\rm rank}\left(\left[\begin{array}{cc}
a+t1&b\\
b&d\end{array}\right]\right)\leq n_0\ \ \mbox{for all}\ t\in \f.
\end{equation}
If $\f$ is infinite, there exists $t\in\f$ such that $a+t1$ is invertible. By elementary column operations (that is, multiplying the matrix in 
(\ref{502}) from the right by a suitable invertible matrix) 
it follows that the rank of the matrix
$$\left[\begin{array}{cc}
a+t1&0\\
c&-c(a+t1)^{-1}b+d\end{array}\right]$$
is at most $n_0$. Since $a+t1$ is of size $n_0$, it follows that $-c(a+t1)^{-1}b+d=0$ for each $t\in\f$ such that
$a+t1$ is invertible. Thus $c(a+t1)^{-1}b=d$ is a constant matrix, which implies that $d=0$. (This is obvious if $\f=\bc, \br$ 
since we may let $t\to\infty$. The general case follows from the formula for the inverse of a matrix using subdeterminants,
namely $(a+t1)^{-1}=(p(t))^{-1}q(t)$, where $q(t)$ is a matrix polynomial of degree at most $n_0-1$, while $p(t):=\det(a+t1)$ has degree $n_0$.
The identity $cq(t)b=p(t)d$ therefore implies that $d=0$.) This proves that $\mu_{i,j,k}=0$ for all $k$ if $i>n_0$ and 
$j>n_0$. Thus $R_1:={\rm span}\{b_i:\, i>n_0\}$ and $R_2:={\rm span}\{c_j:\, j>n_0\}$ are subspaces
of finite codimension in $R$ such that $R_1R_2=0$. Putting $R_0:=R_1\cap R_2$, we have found a subspace of finite codimension in $R$
with $R_0^2=0$. Now, since for each $x\in R$ the space $\{s\in R_0:\ xs\in R_0\}$ is of finite codimension (namely, this is just the kernel of the map
$s\mapsto q(xs)$, where $q:R\to R/R_0$ is the quotient map) and $R=R_0\oplus X$ for a finite dimensional $X$, it follows 
(as in the proof of Lemma \ref{leA}) that there exists a
subspace $S$ of finite codimension in $R_0$ such that $XS\subseteq R_0$. Then $RS\subseteq XS+R_0S=XS\subseteq R_0$, so 
$L:=RS$ is a left ideal of finite codimension
(since $S\subseteq RS$) contained in $R_0$. Similarly we can now find a two-sided ideal $J$ of $R$ of finite codimension, contained in $L$, so $J^2=0$.

When $\f$ is finite the argument of the previous paragraph must be replaced
by the following  induction on $n_0$. We consider the compression of the map $\mu$ to the lower right
corner, that is, the map $\tilde{\mu}:\f^{I}\to\mathbb{M}_{\tilde{I}}(\f)$, where $\tilde{I}=I\setminus\{1,\ldots,n_0\}$ and
$\tilde{\mu}(\rho)$ is the lower right $\tilde{I}\times\tilde{I}$ corner of $\mu(\rho)$ for each $\rho\in\f^I$. 
Let $n_1=\max_{\rho\in\f^I}{\rm rank}(\tilde{\mu}(\rho))$. If $n_1<n_0$, we may inductively assume that there exists 
a subset $I_0$ of $\tilde{I}$ with finite complement such that the compression of $\tilde{\mu}$ to the lower right
$I_0\times I_0$ corner is $0$. Then this holds also for the compression of $\mu$ and the proof can be completed as in the
previous paragraph. If $n_1=n_0$ then, by a suitable replacement of vectors $(a_i)_{i\in\tilde{I}}$ by their 
linear combinations,
we may assume that $n_1=\tilde{\mu}(\lambda_{n_0+1})$ (where, as above, $\lambda_i$ are the coordinate functionals relative to
the basis $(a_i)$). Further, by elementary row and column operations which effect only the coordinates $i\in\tilde{I}$
(that is, by replacing $b_i$ and $c_i$ with suitable linear combinations)
we can achieve that the matrix $\tilde{\mu}(\lambda_{n_0+1})$ has the identity matrix in its upper left $n_1\times n_1$
corner and zeros elsewhere. Then the upper left $n_0\times n_0$ corner of $\mu(\lambda_{n_0+1})$ must be zero since
${\rm rank}(\mu(\lambda_{n_0+1}))\leq n_0$. However, then the matrix $\mu(\lambda_1+\lambda_{n_0+1})$ has rank greater 
than $n_0$ which is impossible by the definition of $n_0$. 
\end{proof}


\begin{thebibliography}{99}


\bibitem{AF} F. W. Anderson and K. R. Fuller, {\em Rings and Categories of Modules} (Second Edition), GTM {\bf 13},
Springer-Verlag, New York, 1992.


\bibitem{A} P. N. \'Anh, {\em Morita duality for commutative rings,} Communications in Algebra {\bf 18} (1990), 1781--1788.

\bibitem{Ar} R. Arens, {\em Operations induced on function classes,} Monatsch. Math. {\bf 55} (1951), 1--19.

\bibitem{Ar1} R. Arens, {\em The adjoint of a bilinear operation,} Proc. Amer. Math. Soc. {\bf 2} (1951), 839--848.

\bibitem{BB} K. I. Beidar, M. Bre\v sar, {\em Extended Jacobson density theorem for rings with derivations and automorphisms,}
Israel. J. Math. {\bf 122} (2001), 317--346.

\bibitem{BS} D. P. Blecher and B. Solel, {A double commutant theorem for operator algebras,} J. Operator Th. {\bf 51} (2004), 435--453.


\bibitem{CL}C.-L. Chuang and T.-K. Lee, {\em The double centralizer theorem for semiprime algebras}, Alg. and Represent. Th., to appear. 

\bibitem{CFu} R. R. Colby, K. R. Fuller, {\em Equivalence and Duality for Module Categories (with Tilting and Cotilting
for Rings),} Cambridge Tracts in Math. {\bf 161}, Cambridge Univ. Press, Cambridge, 2004.

\bibitem{CF} J. Cozzens and C. Faith, {\em Simple Noetherian Rings}, CTM {\bf 69}, Cambridge Univ. Press,
Cambridge, 1975.

\bibitem{CRT} R. S. Cunningham, E. A. Rutter and D. R. Turnidge, {\em Rings of quotients of endomorphism rings of 
projective modules}, Pac. J. Math. {\bf 41} (1972), 647--668.

\bibitem{DLS} H. G. Dales, A. T. -M. Lau, D. Strauss, {\em Banach Algebras on Semigroups and on Their Compactifications,}
Memoirs AMS {\bf 966}, Providence, R. I., 2010.


\bibitem{F} K. R. Fuller, {\em Density and equivalence,} J. Algebra {\bf 29} (1974), 528--550.

\bibitem{HW} H. Hofmeier and G. Wittstock, {\em A bicommutant theorem for completely bounded module
homomorphisms}, Math. Ann. {\bf 308} (1997), 141--154.


\bibitem{La} T. Y. Lam, {\em Lectures on Modules and Rings,} GTM {\bf 189}, Springer-Verlag, New York, 1999.

\bibitem{M} B. Magajna, {\em Bicommutants and ranges of derivations}, Lin. and Multilin. Alg. {\bf 61} (2013),
1161--1180; a Corrigendum to appear ibid. (http://arxiv.org/pdf/1208.3941.pdf.)

\bibitem{P} T. W. Palmer, {\em Banach Algebras and the General Theory of $*$-Algebras, Vol. 1: Algebras and Banach Algebras,}
Encyclopedia of Math. and its Appl. {\bf 49}, Cambridge Univ. Press, Cambridge, 1994.

\bibitem{Ro} L. H. Rowen, {\em Ring Theory}, Vol. 1, Pure and Applied Math.  {\bf 127}, Academic Press, New York, 1988.

\bibitem{Ro0} L. H. Rowen, {\em Graduate Algebra: Noncommutative View}, GSM {\bf 91}, AMS, Providence, R. I., 2008. 

\bibitem{Sa} F. L. Sandomierski, {\em Modules over the endomorphism ring of a finitely generated projective module,}
Proc. Amer. Math. Soc. {\bf 31} (1972), 27--31.

\bibitem{S} B. Stenstr\" om, {\em Rings of Quotients}, GMW {\bf 217}, Springer-Verlag, Berlin, 1975.

\bibitem{Z} B. Zimmermann-Huisgen, {\em Endomorphism rings of self-generators,} Pacific J. Math. {\bf 61} (1975), 587--602.

\end{thebibliography}
\end{document}